\documentclass[12pt]{amsart}
\usepackage{amsmath,amsthm,amsfonts,amssymb,color,tikz,tikz-cd,
    enumitem,verbatim
}
\DeclareMathOperator{\Tr}{{Tr}} 
\DeclareMathOperator{\Nm}{{N}}
\begin{document} 
\newcommand{\XXX}{\message{XXX}.{\bf XXX !}}
\newcommand{\changed}[1]{\textcolor{red}{#1}}
\newcommand{\anfangchanged}{\color{red}}
\newcommand{\finchanged}{\color{black}}
\let\paragraph\S
\renewcommand{\S}{{\mathbb S}}
\newcommand{\modulo}[1]{\text{ mod }#1}
\newcommand{\Ac}{{\mathop A}_\C}
\newcommand{\Hc}{{\mathbb H}_\C}
\newcommand{\B}{{\mathbb B}}
\newcommand{\C}{{\mathbb C}}
\newcommand{\K}{{\mathbb K}}
\renewcommand{\H}{{\mathbb H}}
\newcommand{\N}{{\mathbb N}}
\newcommand{\Oc}{{\mathbb O}_\C}
\newcommand{\Oct}{{\mathbb O}}
\newcommand{\Olo}{{\mathcal O}}
\newcommand{\Q}{{\mathbb Q}}
\newcommand{\Z}{{\mathbb Z}}
\renewcommand{\P}{{\mathbb P}}
\newcommand{\R}{{\mathbb R}}
\newcommand{\rc}{\subset}
\newcommand{\rank}{\mathop{rank}}
\newcommand{\trace}{\mathop{tr}}
\newcommand{\dimc}{\mathop{dim}_{\C}}
\newcommand{\Lie}{\mathop{Lie}}
\newcommand{\Spec}{\mathop{Spec}}
\newcommand{\Auto}{\mathop{{\rm Aut}_{\mathcal O}}}
\newcommand{\Aut}{\mathop{\rm Aut}}
\newcommand{\alg}[1]{{\mathbf #1}}
\newcommand{\tensor}{\otimes}
\newcommand{\dfe}{\overset{\mathrm{def}}{=\joinrel=}}
\newcommand{\coeff}{\mathop{{\it Coeff}}}
\let\Realpart\Re
\renewcommand{\Re}{\mathop{\Realpart e}}
\let\Impart\Im
\renewcommand{\Im}{\mathop{\Impart m}}
\newtheorem{lemma}{Lemma}[section]
\newtheorem{definition}[lemma]{Definition}
\newtheorem*{claim}{Claim}
\newtheorem{corollary}[lemma]{Corollary}
\newtheorem*{Conjecture}{Conjecture}
\newtheorem*{problem}{Problem}
\newtheorem*{SpecAss}{Special Assumptions}
\newtheorem*{example}{Example}
\newtheorem{example-nr}[lemma]{Example}
\newtheorem*{remark}{Remark}
\newtheorem*{observation}{Observation}
\newtheorem*{warning}{Warning}
\newtheorem*{fact}{Fact}
\newtheorem*{remarks}{Remarks}
\newtheorem{proposition}[lemma]{Proposition}
\newtheorem{theorem}[lemma]{Theorem}
\newtheorem{remark-nr}[lemma]{Remark}
\newtheorem{convention}[lemma]{Convention}
\numberwithin{equation}{section}
\def\labelenumi{\rm(\roman{enumi})}
\newcommand{\notiz}[1]{{\tt To do:}{\color{red}\tt #1}{\message{XX: #1}}}
\newcommand{\finalversion}{
  \newcommand{\ignore}[1]{}
  \let\changed\relax
  \let\anfangchanged\relax
  \let\finchanged\relax
  \let\notiz\ignore
}
%
%
\let\oldlabel\label
\title[Slice Regular function]{%
  Slice Regular functions on alternative $*$-algebras: Prescribing zeroes and
  values
  on discrete sets and related extension problems.
}
\author {Cinzia Bisi \& J\"org Winkelmann}
\begin{abstract}
        {\em Slice regular functions} are a generalization of holomorphic
        functions
        where
        alternative real $*$-algebras are considered instead of the
        field of complex numbers.

        For such functions we show that zero sets and function values
        on suitable subsets may be prescribed.

        As a consequence, we show that for any {\em axially symmetric
          domain} there exist  slice regular functions which
        (due to the nature of its zero set) can not be extended
        to a larger such domain.
\end{abstract}
%
\address{%
Cinzia Bisi \\
Department of Mathematics and Computer Sciences\\
Ferrara University\\
Via Machiavelli 30\\
44121 Ferrara \\
Italy
}
\email{bsicnz@unife.it \newline
 ORCID : 0000-0002-4973-1053
}
\address{%
J\"org Winkelmann \\
IB 3-111\\
Lehrstuhl Analysis II \\
Fakult\"at f\"ur Mathematik \\
Ruhr-Universit\"at Bochum\\
44780 Bochum \\
Germany
}
\email{joerg.winkelmann@rub.de\newline
    ORCID: 0000-0002-1781-5842
}
\thanks{
{\em Acknowledgement.}
The two authors were partially supported by GNSAGA of INdAM.
C. Bisi was also partially supported by PRIN \textit{Variet\'a reali e complesse:
geometria, topologia e analisi armonica}. 
}

\maketitle
\tableofcontents

\newcommand{\chapter}[1]{\centerline{{\bf #1 }}}

\section{Introduction}
        {\em Slice regular functions} are a generalization of holomorphic
        functions
        where
        alternative real $*$-algebras are considered instead of the
        field of complex numbers.

        They were first introduced for the case of quaternions in
        \cite{GS2},
        later generalized to other alternative $*$-algebras (see \cite{GP})
        and since widely studied (see e.g.
        \cite{GSSMono,BW-Runge,BC23,BW3,autoo,autoq}).
        The class of alternative $*$-algebras is a very natural
        family of algebras. It contains the complex filed $\C$,
        the algebra of quaternions $\H$ and the algebra of octonions
        $\Oct$, as well as the Clifford algebras which are important
        areas reaching from abstract algebra to physics.

To be more precise, let        
$A$ be a finite-dimensional alternative real $*$-algebra.
Ghiloni and Perotti (\cite{GP}) defined the notions of a pseudosphere $\S_A$
and a  quadratic cone $Q_A$. The ``slice regular functions'' in which we
are interested are defined 
on ``axially symmetric domains''
in $Q_A$.

This is the class of functions considered in this paper.

Classical results in complex analysis state that
for every discrete subset $D$ in a domain $G\subset\C$
the following assertions are true:
\begin{enumerate}
\item
  For every map $\zeta:D\to\C$ there is a holomorphic function
  $f$ on $G$ with $\forall z\in D: f(z)=\zeta(z)$.
  
  \item
  For every map $\zeta:D\to\C^*$ there is a holomorphic function
  $f:G\to\C^*$ with $\forall z\in D: f(z)=\zeta(z)$.
  \footnote{$(ii)$ is less often stated explicitly, but follows
  directly from $(i)$ using the surjectivity of $\exp:\C\to\C^*$.}
\item
  There is a holomorphic function $f$ on $G$ with
  $D=\{z\in G: f(z)=0\}$.
\end{enumerate}

Classically, these results
are used to show: {\em Given a domain $G\subset \C$ there
  is a holomorphic function $f$ on $G$ which can not be extended to any
  larger domain in $\C$.}
For this purpose one constructs a discrete subset $D$ of $G$
such that every boundary point of $G$ is an accumulation point of $D$
in $\C$.

Our goal is to obtain similar results for
slice regular functions
for arbitrary finite-dimensional
real alternative $*$-algebras.

Here we discuss several related questions for
slice regular functions $f$ on an
``axially symmetric domain'' $\Omega_G$ in arbitrary
finite-dimensional alternative real $*$-algebra $A$:

\begin{enumerate}
\item
  For a given discrete family of isolated points and
  ``pseudospheres'' of the form  $x+y\S_A$, can we also find a slice regular
  function with prescribed values at these points?
\item
  If the prescribed values are all non-zero, may we find a
  slice regular function with these prescribed values such that
  there are no zeroes at all?
\item
  Given a discrete family of isolated points and pseudospheres,
  is there a slice regular function which vanishes at this set
  and nowhere outside this set?
\end{enumerate}

To answer this, we introduce the notion of an {\em adapted set}.
(Definition~\ref{def-adapted}).
This is a precise formulation
  of what we mean with ``discrete family of isolated points and
  pseudospheres''
as alluded to above.

Using this notion, our answers are:
\begin{enumerate}
\item
  Given an adapted set $D$ and 
  a function $\zeta:D\to A$ such that
  for every $x+y\S_A\subset D$
  there are $a,b\in A$
  (which may depend on $x,y$) with
  \begin{equation}\label{aff-lin}
  \zeta(x+yI)=(x+yI)a+b \ \forall I\in \S_A
  \end{equation}
  there exists a slice regular function with $f|_D=\zeta$.
  (Theorem \ref{prescribe-values}).
  \footnote{Conversely, the celebrated
      ``representation formula'' implies
    that for every slice regular function $f$ its restriction
    to any $x+y\S_A$ is an affine-linear function
    as in \eqref{aff-lin}.}
\item
  If in addition the adapted set $D$ satisfies
  $D\cap\R=\{\}$, $\zeta$ is locally constant on $D$ and
  \[
  \forall w\in D: \zeta(w)\in A^*
  \]
  then $f$ may be chosen such that
  \[
  \forall w\in \Omega_G: f(w)\ne 0
  \]
  (Corollary~\ref{cor-p-v-i}.)
\item
  For any adapted set $D$ there exists a slice regular function $f$
  such that
  \[
  D=\{w\in\Omega_G:f(w)=0\}
  \]
  (Theorem~\ref{thm-adapted}).
\end{enumerate}

On the other hand, the following is known
(\cite{camshaft}, Theorem~1, \cite{GPS-Edin}):

{\em
  Let $A$ be a division algebra
  (i.e., $A$ is isomorphic to $\H$ or $\Oct$)
    and let $\Omega$ be an axially
  symmetric domain with $\Omega\cap\R\ne\{\}$.
  Then the zero set of any slice-regular function (which is not
  constantly zero) is an adapted set.
}

Note that we consider arbitrary slice regular functions,
including those which are not necessarily ``admissible'' in the sense
of \cite{GP}.
The zero set of an admissible function is always an adapted
set (under the assumption $\Omega\cap\R=\{\}$),
even if $A$ is not a division algebra 
(\cite{GP}).

We show that for every slice regular function $f$ whose stem function
has an invertible value at one point the zero set is
contained in an adapted set (Proposition~\ref{zero-sub-adapted}).

However, without this assumption the zero set need
not be contained in an adapted set as can be seen in
Example~\ref{not-adapted}.

      \subsection{Non-extendability}

      As an application we use the results on zero sets in order to prove
      (Corollary~\ref{c-no-extend}):

{\em For every axially symmetric domain $\Omega_G$
  in the quadratic cone
  $Q_A$
  of a finite-dimensional alternative real $*$-algebra
  $A$ there is a
  slice regular function on $\Omega_G$
  which can not be extended to any larger
  axially symmetric domain.
}

In fact, we prove something stronger:

{\em Given an axially symmetric domain $\Omega$ in $Q_A$ there exists
  a discrete adapted set $D\subset\Omega$ such that every point in the
  boundary $\partial\Omega_G$ of $\Omega_G$ in $Q_A$
  is an accumulation point of $D$.}

\subsection{Alternative algebras}

For our purposes, in \paragraph\ref{altalg}
we gather a couple of basic algebraic facts on finite-dimensional
complex alternative algebras. E.g., we prove that the notions ``zero divisor''
and ``non-invertible'' coincide, even if the algebra
is only alternative and not associative.

  We also study infinite products in  alternative algebras,
  provide criteria for convergence
  (Proposition~\ref{infinite-product-algebra})
  and show, for example, that
 (with the appropriate definitions)
a convergent infinite product has an invertible limit if and only
if all its factors are invertible
(Corollary~\ref{A-infinite-product}).
This will be used in the construction of slice regular functions
without zeroes (Proposition~\ref{nju}).

Many of these auxiliary results are well-known for associative
  algebras, but we need them in the more general context
  of alternative algebras.
\section{Previous work}

\subsection{Previous work: Cohomological methods}

Using cohomological methods, Prezelj and Vlacci proved results on
prescribing function values, zeroes and poles
(\cite{prezelj2024-arxiv}).

Unlike our results, they considered only the quaternionic case.
On the other hand, their results are stronger in the sense that
jets instead of merely values may be prescribed on discrete sets.

\subsection{Previous work: Weierstrass factorization}

In \cite{Weier-fact}, Weierstra\ss\ factorization for {\em entire slice regular functions}
is discussed.
This is related to our work
on prescribing zeroes for slice regular functions, but in the opposite direction
and  limited to the case where $A=\H$ and $\Omega_G=Q_A=A$:
We use infinite products to {\em construct}  functions, while
the Weierstrass decomposition starts with a given function, providing a
representation as infinite product for this function.
However, using their techniques, it is easy to show that
for the case $ \Omega_G=A=\H$ one can prescribe
any adapted set as the zero set for a slice regular function.

\chapter{Preparations}

\section{Preparations}

We mostly follow the setup in \cite{GP}, which may be summarized
as follows.

$A$ is a real finite-dimensional alternative $*$-algebra with $1$.

The conjugation of the $*$-algebra $A$ is denoted as $x\mapsto x^c$.
It is assumed to be an antiautomorphism (i.e.~$(xy)^c=y^cx^c$) and
to fix $\R$ pointwise.

This conjugation induces notions of a ``trace''
$Tr(w)=w+w^c$ and a ``norm'' $N(w)=ww^c$.

Note: For an arbitrary such algebra $A$
these maps $\Tr$, $\Nm$ may have non-real
values, they are maps from $A$ to $A$.

A pseudosphere $\S_A$ and a quadratic
cone $Q_A$ are defined as usual as

\begin{align*}
  \S_A &=\{w\in A: Tr(w)=0, N(w)=1\}\\
\end{align*}
and
\[
Q_A=\R\cup
\{w\in A: Tr(w),N(w)\in\R, 4N(w)>Tr(w)^2\}.
\]

Note that  $\S_A\subset Q_A$ and  $w\in \S_A\implies w^ 2=-1$:

$Tr(x)=0$ implies $x=-x^c$, which in combination with
$N(x)=xx^c=1$ yields $x^ 2=-1$.

Thus
\[
\S_A=\{w\in A: w=-w^c,\ w^2=-1\}.
\]

  Equivalently,
  $\S_A$ may be defined as the set of all $q\in A$ for which there
  exists a homomorphism $\phi$ of real $*$-algebras from $\C$ to $A$
  with $\phi(i)=q$.

There is the following equivalent characterization of $Q_A$
(\cite{GP}, Proposition~3):
\[
Q_A=\{x+yI:x,y\in\R,I\in\S_A\}.
\]

In other words: $w\in Q_A$ iff there is an morphism of real
$*$-algebras
from $\C$ to $A$ with $w$ in its image.

As a consequence,
every non-zero element of $Q_A$ is invertible
(\cite{GP}, Proposition 1 (4)\&(5)).
\footnote{However, if $Q_A\ne A$, there are in general
  invertible elements outside $Q_A$,.}

Hence $Q_A\ne A$ if
$A$ admits zero divisors.

On the other hand, if $A$ is a division algebra, then $Q_A=A$
(\cite{GP}, Example 1.1) and
\[
\S_A=\{q\in A: q^2=-1 \}.
\]

\begin{definition}\label{def-symm}
  A {``\em symmetric domain}'' in $\C$ is an open subset $G$ such that
  \begin{enumerate}
  \item
    $z\in G\ \iff \bar z\in G$ and
  \item
    the quotient $G/\!\sim$ with
    \[
    z\sim w \ \iff\ \left( z=w \text{ or }z=\bar w\right)
    \]
    is connected.
  \end{enumerate}
\end{definition}
    
For such a symmetric
domain $G$ in $\C$ we define
the {``\em associated axially symmetric domain $\Omega_G$''}
as
\begin{equation}\label{def-omega_g}
\Omega_G=\{x+yI:x+yi\in G, I\in\S_A, x,y\in\R\}.
\end{equation}

Observe that  $\Omega_G$ is a connected open subset of $Q_A$.%
\footnote{Domains are usually required to be connected. For our purpose
  the relevant property is the connectedness of $\Omega_G$. Hence we
  define the notion of a ``symmetric domain'' in $\C$
  in such a way that $\Omega_G$
  is connected, allowing the case where $\Omega_G$ is connected
  despite $G$ being disconnected (e.g.~$G=\C\setminus\R$).
  }

Let
\[
G^+=\{z\in G:\Im(z)\ge 0\}.
\]
Note that there is a natural projection map $\pi:\Omega_G\to G^ +$
given by
\[
\pi(x+yI)=x+|y|i.
\]
If $\S_A$ is compact, then $\pi:\Omega_G\to G^ +$ is a proper map.

    \begin{definition}\label{def-adapted}
     Let $A$ be a real finite-dimensional alternative $*$-algebra.
    Let $G$ be a symmetric domain in $\C$,
    $G^ +=\{z\in G:\Im(z)\ge 0\}$ and let $\Omega_G$ be the
    associated axially symmetric domain in $A$
    (as in \eqref{def-omega_g}).
    Let $\pi:\Omega_G\to G^+$ be the natural projection
    $\pi:x+yI\mapsto x+|y|i$.

    A subset $D\subset\Omega_G$ is called {\em adapted}
    if it satisfies the following two properties:
    \begin{enumerate}
    \item
      $\pi(D)$ is discrete in $G^+$.
    \item
      For every $x,y\in\R$
      \begin{enumerate}
      \item
        either $x+y\S_A\subset D$
      \item
        or $\left(x+y\S_A\right)\cap D$ contains at most one element.
      \end{enumerate}
    \end{enumerate}
  \end{definition}

\section{Correspondence between slice regular and stem function}\label{corr}

Let $G\subset\C$ be a symmetric  domain and let $\Omega_G\subset Q_A$
as in \eqref{def-omega_g}.  
There is a well-known correspondence between
(slice) regular functions $f:\Omega_G\to A$ and stem functions
$F:G\to\Ac=A\tensor_\R\C$.
Usually, this correspondence is described as
  \[
  f(x+yI)=F'(x+yi)+IF''(x+yi)
  \]
  for $x,y\in\R$, $I\in\S_A$ and $F',F'':G\to A$
  with $F(z)=F'(z)\tensor 1+F''(z)\tensor i$.

Here we use the following description which is
equivalent to the usual one.

\[
\begin{pmatrix} f(x+yI) \\ f(x-yI)\\
\end{pmatrix}
=
M\cdot
\begin{pmatrix} F(x+yi) \\ F(x-yi)\\
\end{pmatrix},
\quad
\begin{pmatrix} F(x+yi) \\ F(x-yi)\\
\end{pmatrix}
=M\cdot \begin{pmatrix} f(x+yI) \\ f(x-yI)\\
\end{pmatrix}
\]

with
\[
M=\begin{pmatrix} e_1 & e_2 \\ e_2 & e_1 \\
\end{pmatrix},
\quad e_1=\frac 12\left(1-iI\right),\
e_2=\frac 12\left(1+iI\right).
\]

(Observe that $e_1^ 2=e_1, e_2^ 2=e_2, e_1e_2=0, e_2e_1=0, e_1+e_2=1$
and $M\cdot M=E_2$.)

Equivalently:
\begin{align*}
  F(x+yi)&=\frac {1-iI}2f(x+yI)+\frac{1+iI}2f(x-yI)\\
  &=
  \frac {1}2\left(f(x+yI)+f(x-yI) \right) -
  \frac {iI}2\left(f(x+yI)-f(x-yI) \right).
\end{align*}

and vice versa:
\begin{align*}
  f(x+yI)&=\frac {1-Ii}2F(x+yi)+\frac{1+Ii}2F(x-yi)\\
  &=
  \frac {1}2\left(F(x+yi)+F(x-yi) \right) -
  \frac {Ii}2\left(F(x+yi)-F(x-yi) \right).
\end{align*}

Note:

$\frac {1}2\left(f(x+yI)+f(x-yI) \right)$ is also known as
{\em ``spherical value''} and
$\frac {-I}{2y}\left(f(x+yI)-f(x-yI) \right)$ is known
as {``\em spherical derivative''}.

\begin{fact}\label{sphere-constant}
  If $F(x+yi)\in A\subset\Ac$,
  then $f(x+yI)=F(x+yi)\ \forall I\in\S_A$.
  \end{fact}

\subsection{Convergence of slice functions versus convergence
  of stem functions.}

\begin{proposition}
  Let $A$ be a finite-dimensional real alternative $*$-algebra.
  Let $G$ be a symmetric domain in $\C$ with associated axially symmetric
  domain $\Omega_G\subset Q_A$.

  Let $f, f_n$ (with $n\in\N$) be slice regular functions on $\Omega_G$
  with associated stem functions $F,F_n:G\to \Ac$.

  Then the following properties are equivalent:
  \begin{itemize}
    \item
    $f=\displaystyle\lim_{n\to+\infty}f_n$ locally uniformly on $\Omega_G$.
    \item
      $F=\displaystyle\lim_{n\to+\infty}F_n$ locally uniformly on $G$.
  \end{itemize}
\end{proposition}

\begin{proof}
  For $A=\H$ this is proved in
  \cite{BW-Runge} and \cite{MR4078414}, but the proofs
  are easily adapted to the general case as well.
\end{proof}

\section{Alternative algebras}\label{altalg}

The purpose of this section
  is to collect some general facts on alternative algebras which
  we need for our study of slice regular functions.

All algebras $A$ are assumed to be {\em unital}, i.e., there is a unit
$1_A$ with $1_A\cdot x=x=x\cdot 1_A\ \forall x\in A$.

Algebra homomorphisms are assumed to be unital, i.e., taking $1$ to $1$.

An algebra $A$ is called {\em alternative} if
\[
\forall x,y\in A: x(xy)=(xx)y,\quad x(yy)=(xy)y.
\]

For general information about non-associative algebras see
\cite{Schafer}.

An important tool in dealing with alternative algebras is the following
theorem of Artin:
\begin{theorem}\label{artin-thm}
  Let $A$ be an alternative algebra over a field $k$ which is generated
  as a unital $k$-algebra by
  two elements.

  Then $A$ is associative.
\end{theorem}

\begin{proof}
See  \cite{Schafer}, Theorem~3.1.
\end{proof}

As a first preparation, we deduce a lemma on the exponential map for
certain complex algebras.

Let $A$ be a complex algebra which admits a
  complete norm
  (as vector space) satisfying
  \[
  \forall x,y\in A:||xy||\le||x||\cdot||y||.
  \]
  (Such a norm always exists if $\dim(A)<+\infty$, see Lemma~\ref{a-norm}.)
  
Then
the exponential series is convergent and thus defines a map
\[
x\mapsto \exp(x)=\sum_{k=0}^{+\infty} \frac{x^k}{k!}
\]

The proof is the same as in the case of real or complex numbers or matrices.

Furthermore, it is easily seen that
$
\exp(x+y)=\exp(x)\exp(y)
$
if $xy=yx$ (but not in general as is well known
from the case of matrix algebras).

\begin{lemma}\label{exp-image}
  Let $A_0$ be a finite-dimensional
  commutative and associative $\C$-algebra.

  Let $x\in A_0$ be an invertible element.
  Then there exists an element $y\in A_0$
  with $\exp(y)=x$.
\end{lemma}

\begin{proof}
  Since $A_0$ is a finite dimensional commutative
  complex algebra, it is an Artinian ring.
  It follows (see \cite{AM}, Theorem 8.7) that $A_0$ is
  isomorphic to a direct sum of
  finite-dimensional local $\C$-algebras.
  Hence we may assume that $A_0$ is a local $\C$-algebra
  with maximal ideal $m$.
  Since $A_0$ is finite-dimensional, the descending
  sequence of ideals
  $m\supset m^2\supset m^3\supset\ldots$ terminates.
  
  Therefore $m^n=\{0\}$ for some $n\in\N$ and  all elements of $m$
  are nilpotent.

  Furthermore the residue field $A_0/m$ is finite-dimensional
  over $\C$. Since $\C$ is algebraically closed, $A_0/m\simeq\C$
  and $m$ is (as a vector subspace) of complex
  codimension $1$ in $A_0$.
  It follows that every element $x\in A_0$ may be written
  as $x=s+r$ with $s\in\C$ and $r\in m$.
  
  Given $x\in A_0^*$, we write $x=c(1+t)$ with
  $c\in\C^*$ and $t\in m$.
  The logarithmic series
  \begin{align}\label{log-series}
  \log(1+t)=\sum_{k=1}^{+\infty}(-1)^{k+1}\frac{t^k}{k}
  \end{align}
  degenerates to a polynomial, because $t$ is nilpotent.
  Thus $\log(1+t)$ is well-defined. Now choose $z\in\C$ with
  $e^ z=c$ and define $y=z+\log(1+t)$.
  Then
  \[
  \exp(y)=\exp(z)\cdot\exp\left(\log(1+t)\right)=c(1+t)=x.
  \]
\end{proof}

\begin{proposition}\label{alt-eq}
  Let $A$ be a unital finite-dimensional alternative algebra
  over some field $k$.
  
  Fix $x\in A$ and
  let $A_0$ be the $k$-subalgebra generated by $x$ (and $1$).

  Then the following properties are equivalent:
  \begin{enumerate}
  \item
    $x$ is left invertible, i.e. $\exists y\in A:yx=1$,
  \item
    $x$ is right invertible, i.e. $\exists y\in A:xy=1$,
  \item
    $x$ is left invertible in $A_0$, i.e. $\exists y\in A_0:yx=1$,
  \item
    $x$ is right invertible in $A_0$, i.e. $\exists y\in A_0:xy=1$,
  \item
    $x$ is not a left zero divisor,
    i.e., $xy\ne 0\ \forall y\in A\setminus\{0\}$.
  \item
    $x$ is not a right zero divisor,
    i.e., $yx\ne 0\ \forall y\in A\setminus\{0\}$.
  \item
    $x$ is not a left zero divisor in $A_0$,
    i.e., $xy\ne 0\ \forall y\in A_0\setminus\{0\}$.
  \item
    $x$ is not a right zero divisor in $A_0$,
    i.e., $yx\ne 0\ \forall y\in A_0\setminus\{0\}$.
   \item
    $\det(L_x)\ne 0$ where $L_x:A\to A$ denotes the $k$-linear map
    $w\mapsto xw$.
  \item
    $\det(R_x)\ne 0$ where $R_x:A\to A$ denotes the $k$-linear map
    $w\mapsto wx$.

  \end{enumerate}

If $k=\C$, these properties are also equivalent to the following
two assertions.
\begin{enumerate}
\item[\rm (xi)]
    $x$ is in the image of $\exp$, i.e.,
    $\exists z\in A:\exp(z)=x$.
  \item[\rm(xii)]
    $x$ is in the image of $\exp$ restricted to $A_0$, i.e.,
    $\exists z\in A_0:\exp(z)=x$.
\end{enumerate}
\end{proposition}

  \begin{remark}
    For an arbitrary field $k$ there is no exponential map and even if there
    is an exponential map, $\exp:k\to k^*$ is not necessarily
    surjective.
    (For $k=\R$ we have $\exp(\R)=\R^+\ne\R^*$.
    Hence $(i)\ \not \hskip -12pt  \iff(xi)$ for $k=\R=A$.) 

    Therefore the last two properties are equivalent only in the complex
    case.
  \end{remark}

\begin{warning}
   For infinite-dimensional algebras many of these
  properties are inequivalent.
\end{warning}

We prepare the proof of Proposition~\ref{alt-eq} with the
auxiliary lemma below.

\begin{lemma}\label{alt-aux}
  Let $x$ be an element in an algebra $A$ over a field $k$.

  Then:
  \begin{enumerate}
  \item
    If $A$ is finite-dimensional and $x$ is not a right zero divisor, then
    $x$ is left-invertible.
  \item
    If $A$ is finite-dimensional and $x$ is not a left zero divisor, then
    $x$ is right-invertible.
  \item
    If $A$ is associative and $x$ is left-invertible,
    then $x$ is not a left zero divisor.
  \item
    If $A$ is associative and $x$ is right-invertible,
    then $x$ is not a right zero divisor.
  \end{enumerate}
\end{lemma}

\begin{proof}
  \begin{enumerate}
  \item
     If $x$ is not a right zero divisor, then
     $\ker R_x=\{0\}$ (with $R_x:w\mapsto wx$), i.e., $R_x$ is
     injective.
    Because $\dim(A)<+\infty$, this implies
    that the linear endomorphism $R_x$ of
    the vector space $A$ is surjective.
    Hence $\exists y\in A:R_x(y)=yx=1$.
  \item
    Follows by similar arguments.
  \item
    Assume $x,y,z\in A$ with $y\ne 0$, $xy=0$ and $zx=1$.
    Then (using associativity)
    we arrive at a contradiction:
    \[
    0=z\cdot 0 =z\underbrace{(xy)}_0=\underbrace{(zx)}_1y=
    1\cdot y= y\ne 0
    \]
  \item
    Follows by similar arguments.
  \end{enumerate}
\end{proof}
\begin{proof}[Proof of Proposition~\ref{alt-eq}]
  Since $A$ is alternative, we have $(x^n)(x^m)=(x^m)(x^n)$ for all
  $n,m\in\N$. It follows that $A_0$ is a commutative and associative
  algebra.

  Commutativity of $A_0$ yields $(iii)\iff(iv)$.

   To check $(vi)\iff(x)$, we  consider the linear map $R_x:A\to A$
  defined as $R_x(w)=wx$.
  We observe that $x$ is a right zero divisor iff $\ker R_x\ne\{0\}$
  which in turn is equivalent to $\det(R_x)=0$.

  Similarily it follows that $(v)\iff(ix)$.

  $(vi)\implies(i)$ follows from Lemma~\ref{alt-aux}, $(i)$.

  $(i)\implies(vii)$: Assume the contrary. Then
  $\exists y\in A,z\in A_0\setminus\{0\}: yx=1,xz=0$.
  Let $A_1$ be the subalgebra of $A$ generated by $x$ and $y$.
  Note that $A_1$ is associative due to Artin's theorem
  (Theorem \ref{artin-thm}).
  Since $A_0$ is generated by $x$ and $z\in A_0$,
  we have $z\in A_1$. Thus all the three elements $x,y,z$ are
  contained in one associative subalgebra and
  we arrive at a contradiction to Lemma~\ref{alt-aux} $(iii)$.

  $(vii)\implies(iv)$:
  This is due to Lemma~\ref{alt-aux} $(ii)$.
  
  $(iv)\implies (vi)$:
  Assume the contrary. Then $\exists y\in A_0,z\in A\setminus\{0\}:
  yx=1, xz=0$ with $x,y,z$ all contained in an associative subalgebra,
  namely the subalgebra $B$ generated by $x$ and $z$.
  Therefore this contradicts Lemma~\ref{alt-aux}, $(iv)$.

  Thus
    \[
    (vi)\implies (i)\implies (vii)\implies(iv)\implies (vi)
    \]
    
    By similar arguments one deduces
    \[
    (v)\implies (ii)\implies (viii)\implies (iii)\implies (v)
    \]

    From now on assume $k=\C$.

  The implication
  $(xii)\implies(xi)$
  is trivial.

  Since $z$ and $-z$ commute for any choice of $z\in A$, we have
  \[
  \forall z\in A: \exp(-z)\exp(z)=\exp(\underbrace{-z+z}_0)=1
  \]
  Thus an element $x\in A$ admits a left inverse
  if it is contained in the image of the exponential map, i.e.,
  $(xi)\implies (i)$.

  $(iii)\implies(xii)$ follows from Lemma~\ref{exp-image},
  because $A_0$ is commutative and associative.

  As the commutative diagram below shows, the implications
  deduced so far are
  enough to conclude:
  \begin{itemize}
  \item
    For $k=\C$ all twelve assertions are equivalent.
  \item
    For an arbitrary field $k$, the assertions $(i)\sim(x)$
    are equivalent.
  \end{itemize}
  \[
  \begin{tikzcd}
    & (xii)\arrow[rr,Rightarrow]&  & (xi) \arrow[d,Rightarrow]\\
    (ix) \arrow[d,Leftrightarrow] &
    & (vii) \arrow[d,Rightarrow] \arrow[r,Leftarrow]& (i) \\
    (v) \arrow[r,Leftarrow] \arrow[d,Rightarrow] 
    & (iii) \arrow[r,Leftrightarrow]
    \arrow[uu,Rightarrow]& ( iv)\arrow[r,Rightarrow] &
    (vi) \arrow[d,Leftrightarrow]\arrow[u,Rightarrow] \\
    (ii) \arrow[r,Rightarrow] & (viii) \arrow[u,Rightarrow]&& (x) \\
  \end{tikzcd}
  \]
  \end{proof}

\begin{remark-nr}\label{hufeisen}
Since {\em right invertible } and {\em left invertible} are equivalent
properties
for finite-dimensional alternative algebras,
we may and do simply speak about ``invertible elements''.
The set of invertible elements in an algebra $A$ is denoted
as $A^*$.

In a similar spirit, since  an element $x$ in such an algebra
  is a left zero divisor if and only if it is a right
  zero divisor, we speak simply about ``zero divisors''.
  
\end{remark-nr}

\begin{corollary}\label{inv-unic}
  An invertible element in a
  finite-dimensional alternative $k$-algebra has a
  {\em unique} inverse, which is at the same time a left
  and a right inverse.
\end{corollary}

\begin{proof}
  Let $x$ be an invertible element and assume $yx=1$.
  If $z$ is also an inverse, i.e., $zx=1$, then
  $(y-z)x=0$. But $x$ is not a zero divisor
  (Proposition $(i)\implies(vi)$),
  therefore $(y-z)x=0\implies y=z$.

  Since an invertible element $x$ has a left inverse {\em within} the
  subalgebra $A_0$ generated by $x$ (Proposition $(i)\implies(iii)$),
  and $A_0$ is commutative, the
  uniqueness of the left or right inverse implies that left and right
  inverses are equal to each other.
\end{proof}

The following is an improvement on \cite{GPS-TAMS}, Lemma~1.5.

\begin{corollary}\label{inv-prod}
  Let $x,y\in A$ where $A$ is a
    finite-dimensional alternative $k$-algebra.

  Then $xy$ is invertible if and only if both $x$ and $y$ are
  invertible.
\end{corollary}

\begin{proof}
  Let $x,y$ be invertible elements.
  Their inverse elements $x^ {-1}, y^{-1}$ are unique
  (Corollary~\ref{inv-unic}).
  Due to Proposition~\ref{alt-eq} $(i)\iff(iii)$ we know that
  these inverse elements $x^ {-1},y^ {-1}$
  are contained in the subalgebra $A'$ generated  by $x$ and $y$.
  Since $A'$ is generated by two elements,
  it is associative and we have
  \[
  (xy)(y^{-1}x^{-1})=x\underbrace{(yy^ {-1})}_{1}x^ {-1}=1
  \]
  Thus $xy$ is invertible, too.

  Now consider the case where $x,y\in A$ with $x$ being non-invertible.
  Again let $A'$ denote the subalgebra of $A$ generated by $x$ and $y$.
  Note that $x$ is a right zero divisor in $A'$
  (Proposition \ref{alt-eq}, $(i)\implies (viii)$), i.e., there is an
  element $z\in A'$ such that $zx=0$. Since $A'$ is associative,
  it follows that $z(xy)=(zx)y=0$ and therefore $xy$ can not be
  invertible.
\end{proof}

\begin{corollary}\label{inv-polyprod}
  Let $x_1,\ldots,x_n$ be elements in a
    finite-dimensional alternative $k$-algebra.

  Let $p$ denote the product
  \begin{equation}\label{produkt}
  ((\ldots(x_1x_2)x_3)\ldots x_n).
  \end{equation}
  In other words: $p$ is defined by recursion with
  $p_0=1$, $p_{j+1}=p_j\cdot x_{j+1}$ for $n>j\ge 0$  and $p=p_n$.
  
  Then $p$ is invertible if and only if all of the elements $x_i$ are
  invertible.
\end{corollary}

Note that the location of the parentheses in \eqref{produkt}
is relevant, since $A$ is  not assumed to be associative.

\begin{proof}
  Follows from Corollary \ref{inv-prod} by induction on $n$.
\end{proof}

\begin{corollary}\label{inv-group}
  The set of invertible elements $A^*$ is closed under multiplication.
\end{corollary}

\begin{proof}
Follows from Corollary~\ref{inv-prod}.
\end{proof}

Caveat: If $A$ is not associative, then $A^*$ need {\em not} be a group.
It is, however, what is
  called a {\em ``Moufang loop''}.

\begin{corollary}\label{unit-open}
  The set of invertible elements in a finite-dimensional alternative $\C$
  algebra is a dense Zariski open subset. Its complement is a
  complex hypersurface.
\end{corollary}

\begin{proof}
  The set of invertible elements is the complement of the
  complex hypersurface defined as the zero-set of
  the polynomial function
  $x\mapsto\det(L_x)$.
  (Proposition~\ref{alt-eq} $(i)\iff(ii)\iff(ix)$).
\end{proof}

\begin{corollary}\label{unit-suba}
  Let $A$ be a finite dimensional alternative $k$ algebra and let
  $B\subset A$ be a $k$ subalgebra.

  Then $B^*= A^*\cap B$.
\end{corollary}

\begin{proof}
  The inclusion $B^*\subset A^*\cap B$ is obvious.
  
  Let $x\in B$ and let $A_0$ denote the subalgebra of $A$ generated by $x$.
  Then $x\in A_0\subset B$.

  Proposition~\ref{alt-eq}
  $(i)\iff(iii)$ implies $A_0^*=A^*\cap A_0$.
  Thus
  \[
  x\in A^*\cap B \implies x\in A^*\cap A_0=A_0^*\subset B^*
  \]
\end{proof}

\begin{remark}
  The assumption $\dim(A)<+\infty$ is crucial. For example,
  let $A$ be the field of rational functions $\C(X)$ and let $B$
  be the subalgebra $\C[X]$ containing the polynomials. Then
  $B^*=\C^*$, but $A^*\cap B=B\setminus\{0\}$.
\end{remark}

\begin{corollary}\label{unit-subar}
  Let $A$ be a {\em real} finite-dimensional alternative
  algebra with complexification $\Ac=A\tensor_\R\C$.

  Then $A^*=\Ac^*\cap A$.
\end{corollary}

In other words: If $x\in A$ admits an inverse element in $\Ac$,
  this inverse element is already contained in $A$.

\begin{proof}
  The complexification $\Ac$ is also a real algebra and $A$ is a real
  subalgebra of $\Ac$,
  Hence the assertion follows from
  Corollary~\ref{unit-suba}.
%
%
%
\end{proof}

\begin{corollary}\label{inv-ass}
  Let $x,y,z$ be elements in a
    finite-dimensional alternative $k$-algebra.
  Assume $zx=1$.

  Then
  $z(xy)=y$.
\end{corollary}

\begin{proof}
  If $zx=1$, then $z$ is the unique left inverse of $x$ and contained
  in the subalgebra generated by $x$.

  Therefore $x,y$ and $z$ are all elements
  in the subalgebra generated by $x$ and $y$.
  This subalgebra is associative by the theorem of Artin
  (Theorem \ref{artin-thm}).

  Hence
  \[
  z(xy)=\underbrace{(zx)}_1y=y.
  \]
\end{proof}

\section{Norms}
  
\begin{lemma}\label{a-norm}
  Let $A$ be a finite-dimensional $\R$-algebra.

  Then there exists a norm $||\ ||$
  on the $\R$-vector space $A$ such that
  \begin{equation}\label{normprod}
  ||XY||\le ||X||\cdot||Y||\ \ \forall X,Y \in A
  \end{equation}
\end{lemma}

\begin{proof}
  Let $N$ be a norm on $A$ as a $\R$-vector space and define
  $S=\{v\in A:N(v)=1\}$. Now $S$ is compact,
  because $A$ is finite-dimensional.
  The multiplication map $(X,Y)\mapsto X\cdot Y$ is a bilinear map
  and therefore 
  continuous.
  It follows that
  \[
  \exists K>0: \forall X,Y\in S: N(XY)\le K
  \]
  We define $||X||\dfe K\cdot N(X)$, implying $\frac 1K||X||=N(X)$.
  Every element in $A$ may be realized as $rX$ with $r\in\R$, $X\in S$.
  For $r,s\in\R, X,Y\in S$ we have:
  \begin{align*}
    & N(XY)\le K=K\cdot N(X)N(Y)\\
    \implies &
    \frac 1K||XY||\le K\frac1{K^ 2}||X||\cdot||Y||\\
    \implies &||XY||\le ||X||\cdot||Y||\\
    \implies &||rsXY||\le ||rX||\cdot||sY||\\
  \end{align*}
\end{proof}

\begin{remark}
  In general, we can not  achieve equality in \eqref{normprod}
  because $A$ may contain zero divisors,
  and evidently $0=||XY||<||X||\cdot||Y||$ if $XY=0$ but $X,Y\ne 0$.
\end{remark}

\begin{remark} Every $\C$-algebra may also be regarded as $\R$-algebra,
  hence the statement also applies to $\C$-algebras.
\end{remark}

\subsection{Infinite products}

Here we investigate infinite products with values in algebras
which are not necessarily associative.

Due to the lack of associativity it is important
to specify the order in which products are taken.

\begin{convention}\label{inf-prod}
  For a finite sequence $x_1,\ldots, x_m$ in a
  (not necessarily associative) algebra we define
  $\prod_{k=1}^nx_k$  recursively
  by
  $\prod_{k=1}^0 x_k=1$
  and
  \[
  \prod_{k=1}^{n}x_k 
  =\left(
  \prod_{k=1}^{n-1}x_k\right)x_n\quad(\text{for $0<n\le m$})
  \]
In other words:
\[
\prod_{k=1}^n x_k=(\cdots((x_1x_2)x_3)\cdots x_n)
\]
  \end{convention}

\begin{remark}
  In complex analysis (see e.g.~\cite{Rem}),
  usually an infinite product of numbers
  $\prod_{k=1}^{+\infty}  c_k$
  is defined to be convergent, if
  there is a number $N\in\N$ such that the sequence
  $p_n=\prod_{k=N}^nc_k$ converges to an element of $\C^*$.

  This is not suitable for non-associative algebras. For instance,
  without associativity there is no obvious relation between
  convergence of   $p_n=\prod_{k=N}^nc_k$ and convergence
  of $\tilde p_n=\prod_{k=N+1}^nc_k$.

  However, for the algebra of quaternions (which is an associative
  division algebra) this approach is applicable and has been
  used in
  \cite{Weier-fact}.
\end{remark}

\begin{proposition}\label{infinite-product-algebra}
  Let $A$ be a (not necessarily commutative, associative
  or finite-dimensional)
  $\R$-algebra which is simultanously a Banach space
  such that multiplication and addition are continuous and such that
  \[
  \forall x,y\in A: ||xy||\le||x||\cdot||y||.
  \]

  Let $x_k$ be a sequence in $A$ such that
  \[
  \sum^{+\infty}_{k=1} ||x_k|| < +\infty
  \]

  Then the sequence 
  \[
  p_n=\prod_{k=1}^n \left(1+x_k\right)
  \]
  is convergent.
\end{proposition}

This is  well-known for $A=\R$ or $A=\C$ (see e.g.~\cite{Rem}), but
we give a proof to be sure that everything works for arbitrary
algebras, even if $A$ is  neither commutative nor associative.

\begin{proof}
    Convergence of $\sum_k ||x_k||$ implies convergence
  of $\sum_k\log(1+||x_k||)$, because for $x\in\R$ we have
  \[
  \lim_{x\to 0}\frac{\log(1+x)}{x}=1
  \]
  This in turn implies the convergence of
    \[
    \prod_k (1+||x_k||)=\exp\left(\sum_k\log(1+||x_k||)\right).
    \]
    Let $S$ denote the set of finite subsets of $\N$.
    For $N\in\N$ we define
    \[
    S_N=\left\{I\in S: I\subset\{1,\ldots,N\}\right\}.
    \]
    
    Then
    \begin{equation}\label{pnc}
    \prod_{k=1}^{+\infty} (1+||x_k||)=
    \sum_{I\in S} \prod_{k\in I}||x_k||
    =\lim_{N\to+\infty}\sum_{I\in S_N} \prod_{k\in I}||x_k||.
    \end{equation}

    We will consider
    \[
    P_N=\prod_{k\le N} (1+x_k)
    \]
    Here the product is to be understood as in Convention~\ref{inf-prod}.
    In other words, $P_N$ may be defined recursively as
    \[
    P_0=1,\quad P_N=P_{N-1}\cdot\left(1+x_N\right)\quad\text{ for $N>0$}
    \]
    
    We claim that $P_N$ is a Cauchy sequence.
    
    Let $N>n$.
    Using the above introduced notation, by induction one verifies
    \[
    P_N-P_n=
    \sum_{I\in S_N\setminus S_n} \prod_{k\in I}x_k
    \]
    Due to absolute convergence of \eqref{pnc}
    we know that
    \[
    \lim_{n\to+\infty}\sum_{I\in S\setminus S_n} \prod_{k\in I}||x_k||=0
    \]
    Hence there is a sequence of positive numbers
    namely
    \[
    C_n=\sum_{I\in S\setminus S_n} \prod_{k\in I}||x_k||
    \]
    such that ${\displaystyle\lim_{n\to+\infty}} C_n=0$ and
    \[
    \forall N>n: ||P_N-P_n||\le C_n.
    \]
    It follows that $P_n$ is a Cauchy sequence.
    Since $A$ is a Banach space, it is complete.
    Thus every Cauchy sequence (in particular the sequence $P_N$)
    is convergent.
     \end{proof}

%
%
%

The notion of ``normal convergence'' for an infinite product
which is used in complex analysis
(\cite{Rem}), is suitable for our purposes as well. Hence we
introduce the following

\begin{definition}\label{norm-conv}
  Let $X$ be a complex manifold and let $\Ac$ be a
  finite-dimensional complex algebra. Let $f_n:X\to\Ac$
  be a sequence of holomorphic mappings.

  Then $\prod_{n=1}^{+\infty} f_n$ is {\em normally convergent}
  if every compact subset $K$ of $X$ satisfies
  \[
  \sum_{n=1}^{+\infty} ||f_n-1||_K < +\infty
  \]
\end{definition}

Here $||f_n||_K$ denotes the sup-norm induced by a norm on $\Ac$.
The notion is clearly
independent of the choice of the norm on $\Ac$,
 because $\dim(\Ac)<+\infty$.

\begin{corollary}\label{A-infinite-product}
  Let $\Ac$ be a finite-dimensional complex alternative algebra,
  equipped with some norm%
  \footnote{Note: we do not (and in general can not) assume
    that $\forall x,y:||xy||=||x||\cdot||y||$.},
  let $G$ be a domain in $\C$ and 
  let $f_k:G\to\Ac$ be a sequence of holomorphic functions
  such that $\prod_{k=1}^{+\infty} f_k$
  is {\em normally convergent} (Definition~\ref{norm-conv})
  and
  \begin{equation}\label{anker}
    \exists p\in G:
    \exists c\in\Ac^*:
    \lim_{n\to+\infty}\prod_{k=1}^n f_k(p)=c
  \end{equation}
  
  Then there is a holomorphic map $f:G\to \Ac$
  such that
  \begin{enumerate}
  \item
    \[
    \lim_{n\to+\infty}\prod_{k=1}^n f_k(z)=f(z)
    \]
   locally uniformly on $G$
 \item
   \[
   \forall z\in G:\quad
    \left( f(z)\in\Ac^*\right) \ \iff\
    \left( \forall k\in\N: f_k(z)\in\Ac^*\right)
    \]
  \end{enumerate}
\end{corollary}

\begin{proof}
  Fix a compact subset $K\subset G$. Let $A(K)$ denote the vector
  space of holomorphic functions defined in some neighborhood of $K$
  with values in $\Ac$.
  Since $\Ac$ is finite-dimensional, all norms on $\Ac$ are
  equivalent and due to Lemma~\ref{a-norm}
  there is no loss in generality in assuming that
  the norm on $\Ac$ satisfies
  \[
  \forall x,y \in\Ac: ||xy||\le ||x||\cdot||y||
  \]
  We endow $A(K)$ with the corresponding sup-norm
  \[
  ||f||_K=\sup_{x \in K}||f(x)||
  \]
  and the algebra
  structure given by pointwise multiplication and addition.
  Evidently
  \[
  ||fg||_K\le\max_{z\in K} \left(||f(z)||\cdot||g(z)||\right)
  \le ||f||_K\cdot||g||_K
  \]
  Recall that the limit function  of a uniformly converging
  sequence of holomorphic functions is holomorphic.
  Thus $A(K)$ is a Banach  space.
  
  Now we may invoke Proposition~\ref{infinite-product-algebra} and deduce
  that there is a function $g_K:K\to\Ac$ such that
   \[
    \lim_{n\to+\infty}\prod_{k=1}^n f_k(z)=g_K(z)
    \]
    uniformly on $K$. By identity principle these functions $g_K$ patch
    together to one function $f:G\to\Ac$ which is holomorphic
    since all the $f_k$ are holomorphic.
    This proves the first statement.
    
    Next we have to prove the second statement.
    Instead we prove the equivalent statement
   \begin{equation}\label{5.2.2n}
   \forall z\in G:\quad
    \left( f(z)\not\in\Ac^*\right) \ \iff\
    \left( \exists k\in\N: f_k(z)\not\in\Ac^*\right)
   \end{equation}

    We start
    with the $\Leftarrow$-direction.
  
  Invertible elements  in a finite-dimensional alternative $\C$-algebra
  form a Zariski open subset (Corollary~\ref{unit-open}). The
  complement (the set of non-invertible elements) is an analytic subset
  $Z$
  defined by a holomorphic function, namely $\lambda(x)=\det L_x$ where
  $L_x$ denote the endomorphism of $\Ac$ given by $L_x(y)=xy$
  (see Proposition~\ref{alt-eq}, $(i)\iff(xi)$).

  Fix $z\in G$.

  Assume that there is an index $r$ for which $f_r(z)$ is not invertible.
  Consider
  \begin{equation}\label{def-pn}
  p_n(z)=\prod_{k=1}^n f_k(z)
  \end{equation}
  Due to Corollary~\ref{inv-polyprod} for every $n\ge r$ the value
  $p_n(z)$ is contained in the set $Z$ of non-invertible elements.
  Since $Z$ is closed, it follows that
  \[
  f(z)=\lim_{n\to+\infty} p_n(z)\in Z,
  \]
  i.e., $f(z)\not\in\Ac^*$.

  Finally, we show the $\Rightarrow$-part of \eqref{5.2.2n}.
  Let $z\in G$ with  $f(z)\in Z$. By \eqref{anker}, we know that
  $f(G)\not\subset Z$.

  Hence  $f^{-1}(Z)$ is a non-trivial closed analytic
  subset of $G$ and therefore discrete.

  We choose $\varepsilon>0$ such that
  \begin{enumerate}
  \item
    $\overline{B_\varepsilon(z)}=\{w\in G:|z-w|\le\varepsilon\}\subset G$
    and
  \item
    $f^{-1}(Z)\cap\overline{B_\varepsilon(z)}=\{z\}$.
  \end{enumerate}
  
  Define $p_n$ as in \eqref{def-pn}.
  We observe that $f^{-1}(Z)$ is the zero set of $\lambda\circ f$
  with $\lambda(x)=\det L_x$. Hence $z$ is an isolated zero
  of $\lambda\circ f$.
  Since $\lim_{n\to+\infty} p_n=f$  and therefore $\lim_{n\to+\infty} \lambda\circ p_n=\lambda\circ f$,
  a theorem of Hurwitz implies that there
  is an element $q\in B_\varepsilon(z)$ and $k\in\N$ such that
  $ \lambda(p_k(q))=0$ and consequently $p_k(q)\in Z$.

  However, as deduced before, $p_k(q)\in Z$ implies $f(q)\in Z$.
  Since $z$ is the only point of $B_\varepsilon(z)$ which $f$ maps into $Z$,
  it follows that $q=z$. Thus there exists an index $k$ with
  $f_k(z)\in Z$, i.e., $f_k(z)$ is not invertible.
  \end{proof}
\section{Zeroes}

\subsection{Description via stem functions}

\begin{lemma}\label{7.1}
  Let $A$ be an alternative $\R$-algebra, $w,v\in A$ and
  $J\in A$ with $J^2=-1$.
  Consider $p=w+Jv\in A$, $P=w+vi=w\tensor 1 + v\tensor i\in A\tensor\C$.

  Then the following are equivalent.

  \begin{enumerate}
  \item
    $p=0$.
  \item
    $P=(i-J)v$.
  \item
    $(i+J)P=0$.
  \end{enumerate}
\end{lemma}

\begin{proof}
    $(i)\implies (ii)$:
  \begin{align*}
    &p=0\iff w+Jv=0\iff w=-Jv\\
    \implies & P=w+vi=-Jv+vi=(i-J)v
  \end{align*}

  $(ii)\implies(iii)$: Observe that
  $(i+J)(i-J)=i^2-J^2=0$.
  By the theorem of Artin (Theorem \ref{artin-thm})
  the subalgebra generated by $J$ and $v$
  is associative. Hence
  \[
  (i+J)P=(i+J)\left((i-J)v\right)
  =(\underbrace{(i+J)(i-J)}_{=0})v=0
  \]

  $(iii)\implies(i)$:
  We have
  \[
  (i+J)P=0\iff (i+J)(w+vi)=0\iff
  \underbrace{(Jw-v)}_{\in A} + i\underbrace{(w+Jv)}_{\in A}=0
  \]
  Observe that for $v,w\in A$ we have
  \[
  v+iw=0 \iff (v,w)=(0,0).
  \]
  Hence $(iii)$ implies $w+Jv=0$ which is equivalent to $(i)$ since
  $p=w+Jv$.
  \end{proof}

\begin{proposition}\label{5.2}
  Let $G$ be a symmetric domain in $\C$, $A$ a
  finite-dimensional real alternative $*$-algebra.
  Let $f$ be a slice regular function on the axially symmetric domain
  $\Omega_G$ associated to $G$ and let  $F:G\to\Ac$ be the
  associated {\em stem function}. Let $x,y\in\R$ and $J\in \S_A$.

  Then
\[
f(x+yJ)=0 \ \iff\ (1-iJ)F(x+yi)=0
\]
\end{proposition}

\begin{proof}
  First we observe that, thanks to Artins theorem
  (Theorem \ref{artin-thm}),
  the subalgebra of $\Ac$ generated by $J$ is associative.
  Thus we may use associativity for any calculation within
  the sub $\C$-algebra of $\Ac$ generated by $J$.
  
  Let 
  \[
  e_1=\frac 12\left(1-iJ\right),\
e_2=\frac 12\left(1+iJ\right).
\]
Then (see \paragraph\ref{corr}):
\[
f(x+yJ)=e_1 F(x+yi)+e_2 F(x-yi)
\]

Thus the assumption $f(x+yJ)=0$ implies
  \begin{align*}
  e_1f(x+yJ)=0\\
  \iff  &e_1\left(e_1F(x+yi)\right)
  +e_1\left(e_2( F(x-yi)\right)=0\\
\end{align*}
We have
\[
e_1\left(e_1F(x+yi)\right)
=\underbrace{(e_1^2)}_{e_1}F(x+yi)=e_1F(x+yi)
\]
because $A$ is alternative.

Furthermore $e_1,e_2, F(x-yi)$ are contained in the sub $\C$-algebra
of $\Ac$ generated by $J$ and $F(x-yi)$. This subalgebra is
associative due to Artins theorem   (Theorem \ref{artin-thm}).

Hence
\[
e_1\left(e_2( F(x-yi)\right)
=\underbrace{(e_1e_2)}_0 F(x-yi)=0
\]

Therefore
\[
f(x+yJ)=0\ \implies\
e_1 F(x+yi)=0\ \iff\
(1-iJ)F(x+yi)=0
\]

Next we consider the opposite direction and assume that
$(1-iJ)F(x+yi)=0$.
We decompose $F(x+yi)$ as
\[
F(x+yi)=
F'+iF'',\quad F',F''\in A
\]
and deduce
\begin{align*}
  &0=(1-iJ)F(x+yi)
  =(1-iJ)(F'+iF'')\\
  &
  =\underbrace{\left(F'+JF''\right)}_{\in A}
  +\underbrace{\left (iF''-iJF'\right)}_{\in iA}\\
  \implies &
  F'+JF''=0 \ \implies F''=JF'
\end{align*}

Then
(using $e_2e_1=0$, see \paragraph\ref{corr})
\[
f(x+yI)=e_2(F'-iF'')=
e_2(F'-i(JF'))=
e_2\underbrace{(1-iJ)}_{2e_1}F'=0
\]
\end{proof}


\begin{corollary}\label{z-c1}
  If $f(x+yI)=0$, then $F(x+yi)$ is a zero divisor.
\end{corollary}

\begin{corollary}\label{cor-spher}
  A slice regular function vanishes on $x+y\S_A$ if and only if
  $F(x+yi)=0$.
\end{corollary}

\begin{proof}
  Apply the Lemma~\ref{7.1} with $f(x+yI)=p$ and $P=F(x+yi)$.

  If $F(x+yi)=0$, $(iii)\implies(i)$ yields $f(x+yI)=0\ \forall I$.

  For the opposite direction, note that
  \begin{align*}
  &f(x+yI)=0\ \implies (i+I)F(x+yi)=0,\\
  &f(x-yI)=0\ \implies (i-I)F(x+yi)=0
  \end{align*}
  Summing up the two equations
    on the right, we obtain $2iF(x+yi)=0$ which
    implies $F(x+yi)=0$.
\end{proof}

\begin{proposition}\label{zero-div-inv}
  Assume that $A$ is a division algebra
  (i.e., $A$ is isomorphic to $\C$, $\H$ or $\Oct$).

  Then for any given $x,y\in\R$ there exists an $f$-zero on $x+y\S_A$
  if and only if $F(x+yi)$ is a zero divisor.
\end{proposition}

\begin{proof}
  This may be deduced from the preceding considerations.

  However, it is also a special case of Theorem~17 of
  \cite{GP}
  because if $A$ is isomorphic to $\C$, $\H$ or $\Oct$
  the ``normal cone'' $N_A$ of $A$ agrees with all of $A$.
\end{proof}

If $A$ is not a division algebra, it is still true that $f$ has no zeroes
on $x+y\S_A$ if $F(x+yi)$ is invertible
(Corollary~\ref{z-c1}).

But $F(x+yi)$ being a zero divisor does no longer imply the existence of
zeroes: For example, simply take a zero divisor $c\ne 0$ in $A$.
The constant function with value $c$ provides an example where the
stem function has non-invertible values, but the slice regular function
has no zeroes.

\subsection{Star products with slice preserving functions}

\begin{definition}\label{def-sp}
  A slice regular function $f$
  on an axially symmetric domain $\Omega$ is called {\em slice preserving}
  iff its stem function is $\C$-valued.
\end{definition}

\begin{remark}
  
  \begin{enumerate}
  \item
    These functions are called
    {``\em real''} in
    \cite{GP} (Definition~10).
  \item
    The condition is equivalent to the assumption
    that $f(\Omega\cap\C_I)\subset\C_I\ \forall I\in \S_A$.
    (\cite{GP}, Proposition~10).
  \end{enumerate}
\end{remark}

\begin{lemma}\label{lem-f.g}
  Let $f,h$ be slice regular functions
  on an axially symmetric domain $\Omega_G$
  induced by a symmetric domain $G\subset \C$.
  Assume that
  $f$ is {\em slice preserving} (Definition~\ref{def-sp}).
  Then
  \[
  (f*h)(q)=f(q)\cdot h(q)\ \forall q\in\Omega_G
  \]
\end{lemma}

\begin{proof}
  Let $F=F_1+iF_2$ and $H=H_1+iH_2$ be the stem functions for $f$ respective
  $h$ (with $F_1,F_2,H_1,H_2:G\to A$).
  Since $f$ is assumed to be slice-preserving, $F_j$ is $\R$-valued
  ($j\in\{1,2\}$).
  
  Let
  $q=x+yI$ and $z=x+yi$.
  Then
  \begin{align*}
  f(q)h(q)&=
  \left(F_1(z)+IF_2(z)\right)
  \left(H_1(z)+IH_2(z)\right)\\
  &=
  F_1(z)H_1(z)+F_1(z)\left(IH_2(z)\right)
  +\left(IF_2(z)\right)H_1(z)\\
  &+\left(IF_2(z)\right)\left(IH_2(z)\right)\\
  \end{align*}
  On the other hand, the stem function of $f*g$ is $F\cdot G$ and
  therefore
  \[
  (f*h)(q)=F_1(z)H_1(z)-F_2(z)H_2(z)
  +I\left(F_1(z)H_2(z)+F_2(z)H_1(z)\right)
  \]
  Observe that $(IF_2(z))H_1(z)=I(F_2(z)H_1(z))$ because $F_2(z)\in\R$.
  Furthermore $F_1(z)(IH_2(z))=I(F_1(z)H_2(z))$ thanks to $F_1(z)\in\R$.
  Moreover
  $(IF_2(z))(IH_2(z))=F_2(z)(I(IH_2(z))$ due to $F_2(z)\in\R$.
  Finally $I(I(H_2(z)))=(I^2)H_2(z)$, because $A$ is an alternative
  algebra. Therefore $(IF_2(z))(IH_2(z))=-F_2(z)H_2(z)$.

  This completes the proof.
\end{proof}

\begin{proposition}\label{star-zero}
  Let $f,g:\Omega\to A$ be slice regular functions and
  assume that $f$ is {\em slice preserving}
  (as defined above (Definition~\ref{def-sp})).

  Then
  \[
  \forall q:
  (f*g)(q)=0\ \iff\ (f(q)=0\text{ or }g(q)=0)
  \]
\end{proposition}

\begin{proof}
    Using Lemma~\ref{lem-f.g} we obtain
  \[
  \forall q: (f*g)(q)=f(q)\cdot g(q).
  \]

  Therefore the direction
  \[
  (f*g)(q)=0\ \Leftarrow\ (f(q)=0\text{ or }g(q)=0)
  \]
  is obvious and it suffices to prove ``$\Rightarrow$''.
  
  Assume $(f*g)(q)=0$ with $q=x+yJ$
  ($x,y\in\R, J\in\S_A$) and $z=x+yi$.
  Let $F$ and $G$ be the stem functions of $f$ and $g$.
  Then $\forall z:F(z)\in\C$, because $f$ is slice-preserving.
Using Proposition~\ref{5.2} the assumption $(f*g)(q)=0$ implies
  \[
  (1-iJ)\left( F(z)\cdot G(z)  \right)=0
  \]
  Since $F(z)\in\C$  which is central in $\Ac$,
  we may conclude
  \[
  F(z)=0\text{ or } (1-iJ)G(z)=0
  \]
  Now $F(z)=0\implies f(q)=0$.
  On the other hand,
  $(1-iJ)G(z)=0\implies g(q)=0$ due to
  Proposition~\ref{5.2}.
  Thus in any case we may conclude that $f(q)=0$ or $g(q)=0$.
  \end{proof}
\section{Runge}

\begin{definition}
  A subset $K$ in a domain (or complex manifold) $X$ is called
  {\em Runge (in $X$)} if for  every holomorphic function
  $f$ defined in some open neighborhood of $K$ in $X$ there exists a
  sequence of holomorphic functions $f_n$ on $X$
  which converge locally uniformly to $f$ on $K$.
\end{definition}

\begin{remark}
  \begin{enumerate}
  \item
    This condition is equivalent to the following statement:
    {\em
      Let $\Olo(K)$ denote the vector space of continuous functions on $K$
  which can be extended to a holomorphic function in a neighborhood
  of $K$ in $X$.

  Then $\Olo(X)$ is dense in $\Olo(K)$ (for the topology
  of locally uniform convergence).
}
  \item
    If $K$ is compact this is moreover equivalent to:
  {\em
    For every holomorphic function $f$ defined on an open neighborhood
    of $K$ and every $\varepsilon>0$ there exists a holomorphic function
    $g$ on $X$ with $|g(z)-f(z)|<\varepsilon\ \forall z\in K$.
  }
  \end{enumerate}
\end{remark}

\subsection{Runge implies Runge with interpolation}

\begin{proposition}\label{runge-interp}
  Let $X\subset\C$ be a domain,
  $S\subset X$ a finite subset, $K\subset X$ a compact subset.

  Assume that $K$ is Runge in $X$.

  Then for every $f\in\Olo(K\cup S)$ and every $\varepsilon>0$
  there is a function $\tilde f\in \Olo(X)$ with
  $||f-\tilde f||_K<\varepsilon$ and $||f-\tilde f||_S=0$.
\end{proposition}

This is the consequence of a general fact:

\begin{proof}
  The assertion is a special case of Corollary 5.4.5 in
  \cite{F2}:
  $K$ is Runge iff it is $\Olo(X)$-convex.
  A given function $f\in\Olo(K\cup S)$ evidently extends to a
  continuous function on $X$ to which we may apply
  Corollary 5.4.5 of \cite{F2} with $X'=S$ and $f_1=\tilde f$.
\end{proof}

For the convenience of the reader we also provide a more
  direct proof.

\begin{proof}
Let $\C^S$ denote the vector space of $\C$-valued functions on the
finite set $S$. We consider 
the natural evaluation map $\eta:\Olo(X)\to\C^S$.
We choose a section, i.e.,
a linear map $\sigma:\C^S\to\Olo(X)$ with $\eta\circ\sigma=id$.

Since $\C^S$ is finite-dimensional, this
linear map $\sigma$ is continuous.
Let $C$ denote the operator norm of the linear map
  $\sigma$ with respect to the sup-norm on $K$, i.e.,
\[
C=\sup\{||\sigma(v)||_K:v\in\C^S, ||v||_S=1\}
\]

Let $f\in\Olo(K)$ and $\varepsilon>0$ be given.
 We have to show the existence of a holomorphic function
  $\tilde f\in\Olo(X)$ with $||f-\tilde f||_K<\varepsilon$
 and $||f-\tilde f||_S=0$.
 Recall that
   $\Olo(X)$ is dense in $\Olo(K)$ if and only if no connected
   component of $X\setminus K$ is relatively compact in $X$.
   Evidently for any finite set $S\subset X$, the union $K\cup S$
   satisfies this property if and only if $K$ satisfies this property.
   Hence in our situation  $\Olo(X)$ being dense in $\Olo(K)$ implies that
   $\Olo(X)$ is dense in $\Olo(K\cup S)$.
   
Thus we choose a function $g\in\Olo(X)$ with
\[
||f-g||_{K\cup S}<\frac{\varepsilon}{1+C}
\]
Consider $h=(f-g)|_S=\eta(f-g)$ and choose $b=\sigma(h)\in\Olo(X)$.
We observe that
\[
||b||_K =||\sigma(h)||_K
\le C||h||_S=C||f-g||_S
\le C ||f-g||_{K\cup S}< C\frac{\varepsilon}{1+C}.
\]
Define
$\tilde f=g+b$. For every $z\in S$ we have
\[
b(z)=h(z)=f(z)-g(z) \ \implies\ \tilde f(z)=g(z)+b(z)=f(z).
\]
On the other hand
\[
||\tilde f-f||_K=||g+b-f||_K\le ||f-g||_K+||b||_K\le
\frac{\varepsilon}{1+C} + C\frac{\varepsilon} {1+C}=\varepsilon.
\]
\end{proof}

\section{Exhaustions}

  \begin{theorem}\label{rr-exh}
    Let $X$ be a Stein manifold and let $q_n$ (with $n\in\N$) be a 
    discrete sequence without repetitions in $X$.

    Then there exists a strictly plurisubharmonic
    exhaustion function $\rho:X\to\R^+_0$ with
    $\rho(q_n)=n\ \forall n\in\N$.
  \end{theorem}

  \begin{proof}
    Define $D=\{q_n:n\in\N\}$.
    Let $i:X\to\C^N$ be a closed embedding and define $j:X\to\C^ {N+1}$
    as $j(x)=(i(x),0)$.
    Then $j(D)$ is a discrete subset of $\C^ {N+1}$ which is contained
    in the linear hyperplane $\C^ N\times\{0\}$.
    It follows that $j(D)$ is tame in the sense of Rosay and Rudin
    (see \cite{MR929658}, Corollary 3.6 (a)), i.e., there is a holomorphic
    automorphism $\phi$ of $\C^{N+1}$ with $\phi(j(q_n))=(0,\ldots,0;n)$.
    
    Since $\Lambda=\{(0,\ldots,0;\sqrt n):n\in\N\}$ is likewise tame,
    there is a holomorphic automorphism $\psi$ of $\C^{N+1}$
    with
    \[
    \psi(0,\ldots,0;\sqrt n)=(0,\ldots,0;n)\quad\forall n\in\N
    \]
    Now
    \[
    \rho(x)=\left\Vert\psi^{-1}( \phi(j(x)))\right\Vert^ 2
    \]
    is the desired exhaustion function.
  \end{proof}
    
\begin{theorem}\label{psh-runge}
  Let $X$ be a Stein manifold with strictly plurisubharmonic
  exhaustion function
$\rho$.

Let $K_c=\{\rho\le c\}$.
Then every holomorphic function defined in a neighborhood
of $K_c$ can be uniformly approximated on $K_c$ by holomorphic functions
defined on the whole manifold $X$.
\end{theorem}

\begin{proof}
This is Theorem 5.2.8. of \cite{MR1045639}.
\end{proof}

\begin{proposition}\label{adapted-exh}
  Let $G$ be a symmetric domain in $\C$ and let $p_n$ be a
  discrete sequence without repetitions in $G^+=\{z\in G:\Im(z)\ge 0\}$.

  Then there exists a sequence of compact subsets $K_n$ such that
  the following properties hold for every $n\in\N$:
  \begin{enumerate}
  \item
    $K_n$ is Runge in $G$,
  \item
    $K_n$ is contained in the interior of $K_{n+1}$.
  \item
    $\cup_{n\in\N} K_n=G$,
  \item
    $p_n,\bar p_n\in int(K_{n+1})\setminus K_n$
    (where $int(K_{n+1})$ denotes the interior of $K_{n+1}$.)
  \end{enumerate}
\end{proposition}

\begin{proof}
    We consider the sequence $q_n$ in $G$ defined as
  \[
  \forall n\in \N: q_{2n-1}=p_n,\ q_{2n}=\bar p_n
  \]
  Due to Theorem~\ref{rr-exh} there is   
  a strictly plurisubharmonic exhaustion function $\rho$ on $G$
  with
  \[
  \forall n\in\N:\ \rho(q_{2n-1})=\rho(p_n)=2n-1,\quad
  \rho(q_{2n})=\rho(\bar p_n)=2n.
  \]
  We observe that
  \[
  \forall n\in\N:
  2n-\frac 32 < 2n-1=\rho(p_n) <\rho(\bar p_n)=2n < 2(n+1)-\frac 32
  \]
  and define
  \[
  K_n=\left\{x\in G:\rho(x)\le 2n-\frac 32\right\}
  \]
  Note that each $K_n$ is Runge (Theorem~\ref{psh-runge}).
  The other assertions follow from the choice of $\rho$.
  \end{proof}

\section{Complex Analytic Preparation}
\subsection{Constructing holomorphic functions}

\begin{proposition}\label{step}
  Let $G\subset \C$ be a symmetric domain,
  $p=x+yi\in G$,
  $K$ a Runge compact subset in $G\setminus\{p,\bar p\}$,
  $\varepsilon>0$ 
  and let $E\subset G\setminus\{p,\bar p\}$
  be a finite set.
  
  Then there exists a holomorphic function $\phi\in\Olo(G)$
  satisfying
  \begin{enumerate}
  \item
    $\phi(p)=0$.
  \item
    $\phi(\bar p)=-2yi$.
  \item
    $\phi(z)\ne 0\ \forall z\in G\setminus\{p\}$,
  \item
    $||\phi-1||_K<\varepsilon$.
  \item
    $\phi(z)=1\ \forall z\in E$.
  \end{enumerate}
\end{proposition}

\begin{proof}
  Without loss of generality, we assume $\varepsilon<1$.
  
  Consider the connected component $W$ of $G\setminus K$ which contains $p$.

  First we consider the case where $W$ is bounded (in $\C$).
  Since $K$ is Runge in $G$, this implies
  $\overline{W}\cap\partial G\ne\{\}$.
  Let $q\in \overline{W}\cap\partial G$.
  Using the M\"obius transformation
  $z\mapsto 1/(z-q)$, we reduce this case to case where $W$ is unbounded.

  Thus without loss of generality we may assume that $W$ is unbounded.
  
  It follows that there is a continuous curve
  $\gamma:[0,+\infty[\to W$
  with $\gamma(0)=p$ and $\lim_{t\to+\infty}|\gamma(t)|=+\infty$.
  Let $Sp(\gamma)$ denote the image of $\gamma$.
  
  As a consequence, there exists a branch $\lambda(z)$ of $\log(z-p)$ on
  $\Omega=G\setminus Sp(\gamma)$.

  Since $E$ is finite, there is no loss in generality in assuming that
  $\gamma$ avoids $E$, i.e., $E\cap Sp(\gamma)=\{\}$.

  Then
  \[
  \forall z\in\Omega:\exp(\lambda(z))=z-p.
  \]
    Due to the Runge assumption on $K$
  we have Runge with interpolation (compare Proposition~\ref{runge-interp}).
  Hence for every $\delta>0$ there is a
  holomorphic function $h$ on $G$ with
  \[
  |\lambda(z)-h(z)|<\delta\ \forall z\in K.
  \]
  and
  \[
  \forall z\in E: \lambda(z)=h(z).
  \]
  Since $\exp$ is continuous with $\exp(0)=1$,
  we may choose $\delta>0$ such that
  \[
  \forall w\in\C: |w|<\delta \ \implies\ |e^w-1|<\frac\varepsilon 3.
  \]

  Then
  \[
  \bigg\vert
  \exp
  \Big( \lambda(z))-h(z)
  \Big)-1
  \bigg\vert
  <\frac \varepsilon 3
  \quad\forall z\in K
  \]

  We recall that $\forall z\in\Omega:\exp(\lambda(z))=\mu(z)=z-p$.

  We define
  \[
  \psi(z)=(z-p)e^{-h(z)}.
  \]
  and observe that $\psi$ (taken as $\phi$) evidently
  satisfies assertions
  $(i)$ and $(iii)$.
  
  Moreover, since $\forall z\in\Omega:\psi(z)=\exp(\lambda(z)-h(z))$,
  conditions $(iv)$, $(v)$ are satisfied as well.
  
  Condition $(ii)$ still has to be addressed.

  If $\bar p=p$, then $y=0$ and condition $(ii)$ is trivially satisfied.
  Hence from now on we assume $\bar p\ne p$.
  
  Note that in this case $\psi(\bar p)\ne 0$,

    We choose $\alpha\in\C$ such that
  \[
  \exp(\alpha)=-\frac{2yi}{\psi(\bar p)}
  \]
  Since $K$ is Runge in $G$,
  we may apply  Proposition~\ref{runge-interp} (with $S=E\cup\{\bar p\}$)
  and choose a holomorphic function
  $g$ on
  $G$ such that
  \begin{enumerate}
    \item
    $||e^g-1||_K<\frac\varepsilon 3$.
    \item
      $g(\bar p)=\alpha$.
    \item
      $\forall z\in E: g(z)=0$.
  \end{enumerate}
  Finally we set $\phi(z)=e^{g(z)}\psi(z)$.
  By construction
  \[
  \phi(z)=0\ \iff\ \psi(z)=0\ \iff\  z=p
  \]
  This yields $(i)$ and $(iii)$.
  Condition $(ii)$ follows from $g(\bar p)=\alpha$.
  Assertion $(v)$ is due to
  \[
  \forall z\in E: \psi(z)=1, g(z)=0
  \]
    To check $(iv)$, we observe
  \begin{align*}
  \forall z\in K:
  &|\phi(z)-1|=|e^{g(z)}\psi(z)-1|\\
  &=
  |\left(e^ {g(z)}-1\right)\psi(z)+(\psi(z)-1)|<
  \frac\varepsilon3\underbrace{(1+\frac\varepsilon3)}_{<2}
  + \frac\varepsilon3<\varepsilon
  \end{align*}
    \end{proof}

\chapter{Main Results: Prescribed Values}
\section{Prescribing values}

\begin{lemma}\label{lem-1}
  Let $G$ be a symmetric domain in $\C$,
  $\Lambda\subset G^+=\{z\in G:\Im(z)\ge 0\}$ a discrete
  subset and let $\zeta:\Lambda\to\C^N$ be a map such that
  $\zeta(z)\in\R^N\ \forall z\in \Lambda\cap\R$.
  
  Then there exists a holomorphic map $F:G\to\C^N$ such that:
  \begin{enumerate}
  \item
    $\forall z\in \Lambda:\ F(z)=\zeta(z)$.
  \item
    $\forall z\in G: \ \overline{F(\bar z)}=F(z)$
  \end{enumerate}
\end{lemma}

\begin{proof}
  We start with the observation that
  \[
  \tilde \Lambda=\Lambda\cup\{z\in G:\bar z\in \Lambda\}
  \]
  is a discrete subset of the domain $G$.
  Hence every map from $\tilde \Lambda$ to a complex
  vector space may be extended to a holomorphic map
  defined on all of $G$ (see e.g.~\cite{MR1185074}, Theorem~26.7).
  We choose a holomorphic map $F_0:G\to \C^N$ such that
  \begin{itemize}
  \item
    $F_0(z)=\zeta(z)\ \forall z\in \Lambda$.
  \item
    $F_0(\bar z)=\overline{\zeta(z)}\ \forall z\in \Lambda$.
  \end{itemize}
  Then $F:G\to\C^N$ defined as
  \[
  F(z)=\frac 12\left( F_0(z)+\overline{F_0(\bar z)}\right)
  \]
  has the desired properties.
\end{proof}

\begin{theorem}\label{prescribe-values}
  Let $G$ be a symmetric domain in $\C$.
  Let $x_n+y_ni$ 
  be a discrete sequence of distinct elements
  in $G^ +=\{z\in\C:\Im(z)\ge 0\}$.

  Let $A$ be a finite-dimensional alternative real
  $*$-algebra.
  For every $n\in\N$ let $L_n:A\to A$
  be an  (left) affine linear function,
  defined as  $L_n:q\mapsto a_n+qb_n$ for some $a_n,b_n\in A$. 

  Let $Q_A$ be the quadratic cone of $A$ and let $\Omega_G$ be the
  axially symmetric domain in $Q_A$
  associated to $G$. Then there exists a slice regular function $f$ on the 
  $\Omega_G$ such that $f$ and $L_n$ agree on $x_n+y_n\S_A$
  for every $n\in\N$.
\end{theorem}

\begin{proof}
  Let $\Lambda=\{x_n+y_ni:n\in\N\}$.
  We define a map $\zeta:\Lambda\to\Ac$ as 
  \[
  \zeta(x_n+y_ni)=a_n+x_nb_n+iy_nb_n
  \]
  Note that
  \[
  x_n+y_ni\in\R \iff y_n=0
  \implies y_nb_n=0\iff  a_n+x_nb_n+iy_nb_n\in A
  \]
  Hence there is a holomorphic function $F:G\to\Ac$
  such that $F(z)=\zeta(z)\ \forall z\in \Lambda$ and
  $F(z)=\overline{F(\bar z)}\ \forall z\in G$
  (Lemma~\ref{lem-1}).

  By construction, the slice regular function $f$ associated to the stem
  function $F$ satisfies
  \begin{align*}
  \forall n\in \N:\forall I\in\S_A:
  f(x_n+y_nI)
  &=\left(a_n+x_nb_n\right)+Iy_nb_n\\
  &= a_n + \left(x_n+Iy_n\right)b_n\\
  &= L_n\left(x_n+y_nI\right)\\
  \end{align*}
  \end{proof}

\begin{corollary}\label{p-v-cor1}
  Let $D$ be a discrete subset of $\Omega_G$ such that
  for every $x,y\in\R$ the set $D\cap(x+y\S_A)$ has at most one
  element.

  Assume that $\pi(D)$ is discrete in $G^ +$.

  Then every map $\zeta:D\to A$ extends to a slice regular function
  $f:\Omega_G\to A$.
\end{corollary}

\begin{proof}
  We write $D$ as $D=\{x_n+y_nI_n:n\in\N\}$ with
  $x_n\in\R, y_n\in\R^+_0, I_n\in\S_A$, choose
  affine linear functions $L_n:A\to A$ with
  \[
  \forall n\in\N:
  L_n(x_x+y_nI_n)=\zeta(x_n+y_nI_n)
  \]
  and invoke Theorem~\ref{prescribe-values}.
\end{proof}

\begin{remark}
  
  Using Definition~\ref{def-adapted}, the statement of the above
  corollary may be reformulated as follows:
  {\em  For every discrete adapted subset $D\subset \Omega_G$ every
    map $\zeta:D\to A$ extends to a slice regular function $f:\Omega_G\to A$.
  }
  
\end{remark}
\begin{remark}
  If $\S_A$ is non-compact, then $\pi(D)$ may be
  a non-discrete subset of $G$ even if
  $D$ is discrete in $\Omega_G$.
  We will discuss this issue in detail in a later paper
  \cite{cone}.
\end{remark}

\begin{corollary}\label{p-v-cor2}
  Let  $A$ be a finite-dimensional real alternative $*$-algebra.
  Assume in addition that $A$ is a {\em division algebra}.
  Let $D$ be a discrete subset of $\Omega_G$ such that
  for every $x,y\in\R$ the set $D\cap(x+y\S_A)$ has at most two
  elements.

  Then every map $\zeta:D\to A$ extends to a slice regular function
  $f:\Omega_G\to A$.
\end{corollary}

\begin{proof}
  Since $A$ is  a division $*$-algebra, we have
  $A\simeq\C$, $A\simeq \H$
  or $A\simeq\Oct$.

  This implies in particular that $\S_A$ is compact. As a consequence,
  the property of $D$ being discrete in $\Omega_G$ implies that
  \[
  \pi(D)=\{x+yi: x,y\in\R, y\ge 0\ \exists I\in\S_A:x+yI\in D\}
  \]
  is discrete in $G$.
  
  Let $p,q\in x+y\S_A$ and let $\alpha:\{p,q\}\to A$ be any map.
  Assume $p\ne q$.
  Since $A$ is a division algebra, $(q-p)^{-1}$ exists, and we may
  define
  \begin{align}
    L(w)&=\alpha(p)+(w-p)(q-p)^{-1}\left(\alpha(q)-\alpha(p)\right)
    \label{eqLdef} \\
  &=w\left( (q-p)^{-1}\left(\alpha(q)-\alpha(p)\right)\right)
    +\alpha(p)-p\left( (q-p)^{-1}\left(\alpha(q)-\alpha(p)\right)\right)
    \nonumber
  \end{align}
  Thus every map $\alpha:\{p,q\}\to A$ extends to an affine-linear
  map.
  Consequently  the assertion of the corollary follows from 
  Theorem~\ref{prescribe-values}.
\end{proof}
  
\begin{remark}\label{ex-aff}
  If $A$ is not a division algebra, the conclusion of Corollary~\ref{p-v-cor2}
  may fail.
  
  The background: $A$ is a division algebra
  if and only if the following property is satisfied:
  {\em Given $p,q,a,b\in A$ with $p\ne q$ there exists an
    affine linear function $L(x)=x\alpha+\beta$ (with $\alpha,\beta\in A$)
    such that $L(p)=a$ and $L(q)=b$.}

  In fact this condition translates into the following linear system
  of equations for $\alpha,\beta$:
  \begin{align*}
    &  p\alpha +\beta =a,\quad q\alpha+\beta=b\\
    \iff &  (p-q)\alpha=a-b,\quad \beta=a-p\alpha\\
  \end{align*}
  Now $(p-q)\alpha=a-b$ is solvable for all $a,b\in A$ iff
  $p-q$ is invertible.

  As a consequence, the conclusion of Corollary~\ref{p-v-cor2}
  fails, if there are $I,J\in\S_A$ such that $I-J$ is
  a zero divisor.

  For example, this happens, if $I^2=-1=J^2$ and $IJ=JI$:
  Then $(I+J)(I-J)=0$. To see a concrete example, consider
  $A=\C\oplus\C$ as a real algebra with $I=(i,i)$ and $J=(i,-i)$.
\end{remark}

\section{Prescribing non-zero values}

\begin{proposition}\label{prescribe-val-inv-dis}
  Let $A$ be a finite-dimensional alternative real $*$-algebra
  with a discrete adapted set $D\subset \Omega_G\subset Q_A$.

  Let $\zeta:D\to A^*$ be any function.

  Assume $D\cap\R=\{\}$.

  Then there exists a slice regular function $f:\Omega_G\to A$
  with $f|_D=\zeta$ and
  \[
  \forall w\in\Omega_G: f(w)\ne 0
  \]
\end{proposition}

\begin{proof}
  Let $D=\{p_k:k\in\N\}$, with $p_k=x_k+y_kI_k$
  (with $x_k, y_k\in\R, I_k\in\S_A$).
  Since $A^*\subset\Ac^*$, for every $k\in\N$
  we may choose an element $c_k\in\Ac$ such that $\exp(c_k)=\zeta(p_k)$
  (Proposition~\ref{alt-eq} $(i)\implies (ix)$).%
  \footnote{In general, it is not possible to choose $c_k$ inside $A$.}

  We required $D\cap\R=\{\}$.
  Hence $\forall k:y_k\ne 0$ and
  (using Lemma~\ref{lem-1}) we may find a holomorphic
  stem function $H:G\to\Ac$ with
  \[
  H(x_k+y_ki)=c_k,\quad
  H(x_k-y_ki)=\overline{c_k}\quad \forall k\in\N.
  \]
  
  Then we define $F(z)=\exp(H(z))$ which is a stem function, too.
  Due to the implication $(ix)\implies(i)$
  of Proposition~\ref{alt-eq}  we know that
  $F(z)\in\Ac^*\ \forall z\in G$.
  Finally we let $f$ be the slice regular function associated to the
  stem function $F$.
  Since $F(z)$ is invertible for all $z\in G$,
  the implication $(i)\implies(v)$ of Proposition~\ref{alt-eq}  and
  Corollary~\ref{z-c1} imply $\forall w\in\Omega_G:f(w)\ne 0$.
\end{proof}
\begin{remark}
    The above result yields a slice regular function $f$ with values
  in $A\setminus\{0\}$, starting from a map $\zeta:D\to A^*$.
  Note that $A^*\ne A\setminus\{0\}$ unless $A$ is a division
  algebra.
  Therefore it is natural to ask whether the above result may
  be strengthened in order to obtain a slice regular function
  $f$ with values in $A^*$ instead of $A\setminus\{0\}$.

  However, this is not always possible. To give an example, we start
  by observing that $A^*$ may be disconnected for an arbitrary $A$.
  For instance, the Clifford algebra $Cl_{1,1}$
  (often denoted $\R_{1,1}$) is isomorphic
  to $Mat(2\times 2,\R)$ and $GL_2(\R)$ is disconnected,
  since for a real matrix the determinant may be positive or negative.
  
  Now fix a real alternative $*$-algebra
  $A$ for which $A^*$ has at least two connected components.
  It is clear that in this case
  not every map $\zeta:D\to A^*$ may be extended to a slice regular
  function $f:\Omega_G\to A^*$:
    As a continuous map, $f$ must map connected components of $\Omega_G$
  to connected components of $A^*$. Hence such an $f$ can not exist, if
  there are $p,q\in D$ in the same connected component of $\Omega_G$
  for which $\zeta(p)$ and $\zeta(q)$ lie in different connected
  components of $A^*$.
\end{remark}

  In fact there is no need to require $D$ to be discrete
  in Proposition~\ref{prescribe-val-inv-dis} above.
  
  \begin{corollary}\label{cor-p-v-i}
    Let $D$ be an adapted set (in the sense of Definition~\ref{def-adapted})
    with $D\cap\R=\{\}$ in $\Omega_G$.

    Then every locally constant
    map $\zeta:D\to A^*$ extends to a slice regular function
    $f:\Omega_G\to A\setminus\{0\}$.
  \end{corollary}

  \begin{proof}
    From every sphere $x+y\S$ contained in $D$ we pick one point and thereby
    obtain a discrete adapted set $\tilde D$ such that
    \[
    \bigcup_{x+yI\in\tilde D}(x+y\S)=
    \bigcup_{x+yI\in  D}(x+y\S)
    \]
    Now we apply Proposition~\ref{prescribe-val-inv-dis}
    obtaining a slice regular function $f:\Omega_G\to A\setminus\{0\}$
    with $\forall z\in\tilde D:\ f(z)=\zeta(z)$.
    The construction in the proof of Proposition~\ref{prescribe-val-inv-dis}
    implies that $f$ is constant
    along $x+y\S$.
    
    Hence $\forall z\in\tilde D:\ f(z)=\zeta(z)$ implies
    $\forall z\in D:\ f(z)=\zeta(z)$ and the proof is completed.
\end{proof}
  
  For division algebras we can say more.
    
\begin{proposition}\label{prescribe-values-inv}

  Let $A$ be a finite-dimensional division algebra.
    Let $G$ be a symmetric domain in $\C$ and let $x_n+y_ni$
  ($n\in\N$)
  be a discrete sequence
  in $G^+=\{z\in G:\Im(z)\ge 0\}$.

  For $n\in\N$ let $L_n:q\mapsto a_n+qb_n$
  be an affine linear function ($a_n,b_n\in A$).

  Assume that
  \[
  \forall n\in\N:
  \forall I\in\S_A:
  L_n(x_n+y_nI)\ne 0
  \]
  
  Then there exists a slice regular function $f$ on the axially symmetric domain
  $\Omega_G$ such that $f$ and $L_n$ agree on $x_n+y_n\S_A$
  for every $n\in\N$ and furthermore $f(q)\ne 0\ \forall q\in\Omega_G$.
\end{proposition}

\begin{proof}
  We define
  \[
  c_n=a_n+x_nb_n+iy_nb_n
  \]
  
  As explained in the proof of Theorem~\ref{prescribe-values}
  these elements $c_n\in\Ac$ satisfy:

  {\em If $F$ is a stem function with $F(x_n+y_ni)=c_n$, then
    the associated slice regular function $f$ agrees with $L_n$ on
    $x_n+y_n\S_A$ for any $n\in\N$.}

  Since $L_n$ is assumed to have no zeroes on $x_n+y_n\S_A$, we know
  that $c_n$ is
  not a zero divisor (Proposition~\ref{zero-div-inv})
  for every $n$. Hence there are elements
  $d_n\in\Ac$ with $c_n=\exp(d_n)$
  (Proposition~\ref{alt-eq} $(v)\iff(ix)$). 
  Now we choose a stem function $H$
  with
  \[
  H(x_n+y_ni)=d_n\ \forall n\in\N
  \]
  and define $F(z)=\exp\left(H(z)\right)$.
  By construction we have $\forall n:F(x_n+y_ni)=c_n$ which implies
  that the slice regular function $f$ associated to the stem function $F$ satisfies
  $\forall n\in \N:\ f|_{x_n+y_n\S}=L_n|_{x_n+y_n\S}$.
\end{proof}

\chapter{Main Results: Prescribed Zeroes}

\section{Prescribing Spherical Zeroes}

\begin{proposition}\label{pr-spher}
  Let $A$ be a finite-dimensional real alternative $*$-algebra
  with quadratic cone $Q_A$,
  $G\subset\C$ a symmetric domain, $\Omega_G\subset Q_A$ the associated
  axially symmetric domain.

  Let  $G^+=\{z\in G:\Im(z)\ge 0\}$ and let $\Lambda$ be a discrete
  subset of $G^+$.

  Then there is a slice regular function $f$  on $\Omega_G$
  such that
  \begin{enumerate}
  \item
    For every $w=x+yi\in \Lambda$, the slice regular function $f$ has
    a spherical zero on $x+y\S_A$, i.e., $f(q)=0\ \forall q\in x+y\S_A$.
  \item
    $f$ has no other zeroes.
  \item
    $f$ is slice-preserving.
  \end{enumerate}
\end{proposition}

\begin{proof}
  Let $F_0:G\to\C$ be a holomorphic function with
  \[
  \{z\in G: F_0(z)=0\}=\Lambda.
  \]
  Then we define $F:G\to\C\subset\Ac$  as
  \[
  F(z)=F_0(z)\overline{F_0(\bar z)}
  \]
  and let $f$ be the slice regular function associated to the stem function $F$.
  For $z=x+yi\in \Lambda$ we have $F(z)=0$ and therefore a spherical zero
  on $x+y\S_A$ (Corollary~\ref{cor-spher}).
  For $z=x+yi\in G\setminus \Lambda$, the value $F(z)$ of the stem function
  satisfies $F(z)\in\C^*\subset\Ac^*$ and thus is not a zero divisor.
  Hence $f$ has no zeroes on $x+y\S$ for $x+yi\not\in \Lambda$.
  (Corollary~\ref{z-c1}).
  
  Note that $f$ is slice-preserving because of $F(G)\subset \C$.
\end{proof}

  \section{Prescribing isolated zeroes}

  \begin{proposition}\label{p-2}
    Let $G$ be a symmetric domain in $\C$, $A$ a
    finite-dimensional real alternative $*$-algebra
    with complexification $\Ac$
    and let $p=x+yi\in G$
    (with $x,y\in\R$), $I\in\S_A$, $\varepsilon>0$.
    Let $E\subset G\setminus\{p,\bar p\}$ be a finite set and $K\subset G$
    be a compact Runge subset with
    $p,\bar p\not\in K$.

    Then there exists a holomorphic function $F:G\to \Ac$ such that
    \begin{enumerate}
    \item
      $F$ is a stem function, i.e., $\overline{F(\bar z)}=F(z)$.
    \item
      $F(p)=y(i-I)$,
    \item
      $\forall z\in E:F(z)=1$.
    \item
      $\forall z\in G\setminus\{p,\bar p\}:\ F(z)\in\Ac^*$.
     \item
       $\forall z\in K:||F(z)-1||<\varepsilon$.
    \end{enumerate}
  \end{proposition}

  \begin{proof}
        Let $E'=E\cup\{\bar z: z\in E\}$.
    Due to Proposition~\ref{step} there is
    a holomorphic function $\phi\in\Olo(G)$
  satisfying
  \begin{enumerate}
  \item
    $\phi(p)=0$.
  \item
    $\phi(\bar p)=-2yi$
  \item
    $\phi(z)\ne 0\ \forall z\in G\setminus\{p\}$,
  \item
    $||\phi-1||_K<\frac 12\varepsilon$.
  \item
    $\phi(z)=1\ \forall z\in E'$.
  \end{enumerate}

  Since $\C$ is an associative real $*$-algebra,
  we may regard $\phi:G\to\C$ as
  a slice regular function $\phi:G\to B=\C$
  with associated stem function
  \[
  \Phi:G\to B_\C=\C\tensor_\R\C
  \]
  We observe that $\phi$ vanishes only in $p$. Therefore
  \begin{equation}\label{eq-p31}
  \forall z\ne p,\bar p: \Phi(z)\in B_\C^* 
  \end{equation}
  due to Proposition~\ref{zero-div-inv}.
  
  The axially symmetric domain in $B$ associated
  to the symmetric domain $G\subset\C$ is isomorphic to $G$
  via the isomorphism $\C\simeq B$. By abuse of language, it
  is again denoted by $G$.
  
  Let $i$ denote the imaginary unit in the first copy $\C$ in $\C\tensor_\R\C$
  (i.e., in $B\simeq\C$) 
  and let
  $\iota$ denote the imaginary unit in the second copy of $\C$ which we
  denote as $\C_\iota$.

  Thus we have a slice regular function $\phi:G\to B$
    and a stem function $\Phi:G\to B\tensor\C_\iota$.
  
  Then (see \paragraph\ref{corr}):
  \[
  \forall u,v\in\R:\Phi(u+v\iota)
  =\frac{1-\iota i}2\phi(u+vi)
  +\frac{1+\iota i}2\phi(u-vi),
  \]
  implying $\Phi(x+y\iota)=y(\iota-i)$, $\Phi(z)=1\ \forall z\in E$
  and
  \[
  \forall z=u+\iota v\in K:
  |\Phi(z)-1 |\le |\phi(z)-1|+|\phi(\bar z)-1|<\varepsilon
  \]

  Now $i\mapsto I$ defines an algebra embedding $\xi$
  (resp.~$\xi_\C$)  of $B$ into $A$ and
  of $B_\C$ into $\Ac$.
  Composing the stem function $\Phi$ with this embedding
  $\xi_\C:B_\C\to\Ac$
  yields a function $F:G\to\Ac$ which is evidently a stem function,
  too.

  The properties of $\Phi$ deduced above immediately imply
  $(ii)$, $(iii)$ and $(v)$.

  Finally we show $(iv)$.
  Due to \eqref{eq-p31} we have
  $\Phi(z)\in B_\C^*$ for $z\in G\setminus\{p,\bar p\}$.
  This implies
  $\forall z\in G\setminus\{p,\bar p\}:\
  F(z)=\xi_{\C}(\Phi(z))\in\Ac^*$, because
  $\xi_{\C}(B_\C^*)\subset \Ac^*$.
  \end{proof}

  \begin{proposition}\label{nju}
    Let $G$ be a symmetric domain in $\C$, $A$ an
    finite-dimensional real alternative $*$-algebra,
    and let $I_n$ be a sequence in $\S_A$.
    Let $p_n=x_n+y_ni$ be an infinite discrete sequence
    without repetitions in
    $G^ +\setminus\R=\{z\in G:\Im(z)> 0\}$.
    Let $\Lambda=\{x_n+y_ni:n\in\N\}\cup\{x_n-y_ni:n\in\N\}$.
    
    Then there exists a holomorphic map $F:G\to\Ac$ such that
    \begin{enumerate}
    \item
      $F$ is a {\em stem function}, i.e., $F(z)=\overline{F(\bar z)}$.
    \item
      $\forall n\in\N: F(x_n+y_ni)=y_n(i-I_n)$
    \item
      $F(z)$ is an invertible element of the algebra $\Ac$ for all
      $z\in G\setminus \Lambda$.
    \end{enumerate}
  \end{proposition}

  \begin{proof}
    We fix an element $c\in G\setminus \Lambda$.

    Due to Proposition~\ref{adapted-exh}
    there is an increasing  sequence of
    compact Runge subsets $K_n\subset G$ such that
      $\forall n: p_n,\bar p_n\in int(K_n)\setminus K_ {n-1}$,
      where $int(\ )$ denotes the interior.

    Next,  using Proposition~\ref{p-2}, we choose recursively
    functions $H_n:G\to\Ac$ such that
    \begin{enumerate}[label=(\arabic*)]
    \item
      $H_n$ is a stem function for all $n\in\N$.
    \item
      $\forall n\in \N:
      H_n(x_n+y_ni)=y_n(i-I_n)$.
    \item
      $\forall n\in \N:\ \forall k<n: H_n(x_k+y_ki)=1$.
    \item
      $\forall n\in \N:\forall z\in G\setminus\{x_n+y_ni,x_n-y_ni\}
      :\ H_n(z)\in\Ac^*$.
    \item
      $\forall n\in\N: H_n(c)=1$.
    \item
      $\forall n\in \N:\ \forall z\in K_n:||H_n(z)-1||<2^{-n}$
    \end{enumerate}

    The last condition ensures that $\prod_{n=1}^{+\infty} H_n$
    converges (Corollary~\ref{A-infinite-product}, statement $(i)$) and thus
    defines a function $H:G\to\Ac$.

    For every $z\in G\setminus \Lambda$ we have $H_n(z)\in\Ac^*\ \forall n$.
    In addition,
    we have $H(c)=1\in\Ac^*$, since $\forall n:H_n(c)=1$.

    Due to $(4)$ and statement $(ii)$ of
    Corollary~\ref{A-infinite-product} it follows that
    $H(z)\in\Ac^*\ \forall z\in G\setminus \Lambda$.

    For every $k\in\N$ we have  $H_n(x_k+y_ki)=1$ for $n>k$.
    Hence
    \[
    H(x_k+y_ki)=\prod_{n=1}^{+\infty}H_n(x_k+y_ki)
    =\prod_{n=1}^{k}H_n(x_k+y_ki)
    \]
    Thus
    \[
    H(x_k+y_ki)=\beta_k H_k(x_k+y_ki), \quad \beta_k=\prod_{n=1}^{k-1}
    H_n(x_k+y_ki)
    \]
    Condition $(4)$ above together with
    Corollary~\ref{inv-polyprod} imply that
    $\beta_k$ is invertible.

    We choose $\gamma_k\in\Ac$ such that $\exp(\gamma_k)=\beta_k$
    (possible by Proposition~\ref{alt-eq}, $(i)\iff(ix)$).

    In combination with $H_k(x_k+y_ki)=y_k(i-I_k)$ this
    yields
    \[
    H(x_k+y_ki)=\exp(\gamma_k)\left(y_k(i-I_k)\right)
    \]
    Due to Corollary~\ref{inv-ass} this implies
    \[
    \exp(-\gamma_k) H(x_k+y_ki)= y_k(i-I_k)
    \]
    
    Next (using Lemma~\ref{lem-1})
    we choose a stem function $M:G\to\Ac$ such that
    \[
    \forall k\in \N: M(x_k+y_ki)=-\gamma_k
    \]
    and define $\tilde M(z)=\exp(M(z))$.
    
    Now
    \[
    \forall k\in \N: \tilde M(x_k+y_ki)=\exp(-\gamma_k)
    \]
    Next we define $F:G\to\Ac$ as
    \[
    F(z)=\tilde M(z)H(z)
    \]
    For every $k\in\N$ we have
    \[
    F(x_k+y_ki)=\tilde M(x_k+y_ki)H(x_k+y_ki)=
    \exp(-\gamma_k) H(x_k+y_ki)=y_k(i-I_k)
    \]
    
    For every $z\in G\setminus \Lambda$, we have $F(z)\in\Ac^*$ due to
    $\tilde M(z),H(z)\in\Ac^*$ and Corollary~\ref{inv-prod}.
  \end{proof}

  \begin{theorem}\label{prescribe-is-zero}
    Let $G$ be a symmetric domain in $\C$, $A$ an
    finite-dimensional real alternative $*$-algebra
    with 
    pseudosphere $\S_A$.
    Let $I_n$ be a sequence in $\S_A$.
    Let $p_n=x_n+y_ni$ be a discrete sequence
    without repetitions in
    $G^ +\setminus\R=\{z\in G:\Im(z) > 0\}$.

    Let $\Omega_G$ be the axially symmetric domain associated to $G$.

    Then there exists a slice regular function $f$ on $\Omega_G$
    such that:
    \begin{enumerate}
    \item
      $f$ has  isolated zeroes on $\{x_n+y_nI_n:n\in\N\}$.
    \item
      $f$ has no other zeroes.
    \item
      $F(z)$ is an invertible element of $\Ac$ for all $z$ in a
      non-empty open subset of $G$.
    \end{enumerate}
  \end{theorem}

  \begin{proof}
    Due to Proposition~\ref{nju} we have a stem function $F:G\to\Ac$
    such that
        \begin{enumerate}
    \item
      $\forall n\in\N: F(x_n+y_ni)=y_n(i-I_n)$.
    \item
      $F(z)$ is an invertible element of the algebra $\Ac$ for all
      $z\in G\setminus \Lambda$
      with $\Lambda=\{x_n+y_ni:n\in\N\}$.
    \end{enumerate}
Let $f$ be the associated slice regular function.
        For a given $n\in\N$ and $J\in \S_A$, we have
        $f(x_n+y_nJ)=0$ if and only if
        $(1-iJ)F(x_n+y_ni)=0$
        (Proposition~\ref{5.2}).
        Now
        \[
        (1-iJ)F(x_n+y_ni)=(1-iJ)y_n(i-I_n)=y_n(i-I_n+J+iJI_n)
        \]

        For $J=I_n$, we have $i-I_n+J+iJI_n=0$ and therefore
        $f(x_n+y_nI_n)=0$.

        Next we show that there are no zeroes outside
        $\{x_n+y_nI_n:n\in\N\}$.
        
        Assume that $J\in\S_A$ with $f(x_n+y_nJ)=0$.
        Then
        \begin{align*}
          &(1-iJ)F(x_n+y_ni)=0\\
          \implies
        &(1-iJ)y_n(i-I_n)=0\\
        &(1-iJ)(i-I_n)=0\\
        \end{align*}
        (Note that by assumption $y_n>0$.)
        Since $i(i+I_n)(i-I_n)=0$, it follows that
        $(-1+iI_n)(i-I_n)=0$ and consequently
        \[
        \left(
        (1-iJ)+(-1+iI_n)
        \right)(i-I_n)=0
        \]
        Hence
        \[
        0=i(I_n-J)(i-I_n)=\underbrace{(J-I_n)}_{\in A}
        +\underbrace{i(1+JI_n)}_{\in iA}
        \]
        which implies
        \[
        J-I_n=0=1+JI_n.
        \]
        Thus we see: $f(x_n+y_nJ)=0$ holds
        only if $J=I_n$.

        Finally,  we show that $f(x+yJ)\ne 0$ for
        \[
        x+yJ\in \Omega_G
        \setminus\bigcup_{n\in\N}
        \left(x_n+y_n\S_A\right)
        \]
        Indeed, $F(x+yi)\in\Ac^ *$. 
        Hence $f(x+yJ)\ne 0$ due to Corollary~\ref{z-c1}.
  \end{proof}

  \section{Prescribing zeroes: mixed case}

  \begin{theorem}\label{thm-adapted}
   
    Let $D$ be an adapted subset of an axially symmetric subset $\Omega_G$
    in a finite-dimensional real alternative $*$-algebra $A$.

    Then there exists a slice regular function $f$
    on $\Omega_G$ and a point $p\in G$ such that
    \[
    D=\{q\in\Omega_G:f(q)=0\}
    \]
    and $F(p)\in\Ac^*$ for the associated stem function $F$.
  \end{theorem}

  \begin{proof}
    Let $D_0$ be the union of all $x+y\S_A$ with $x,y\in\R$, $y\ge 0$ which are
    contained in $D$.
    Due to Proposition~\ref{pr-spher} there is a slice preserving
    function $f_0$ with
    \[
    D_0=\{q\in\Omega_G: f_0(q)=0\}.
    \]
    Let $D_1=D\setminus D_0$. From Definition~\ref{def-adapted}
    it follows that for every $x,y\in\R$ the intersection
    $(x+y\S_A)\cap D_1$ contains at most one element and that
    $\pi(D_1)$ is discrete in $G^+$.
    Theorem~\ref{prescribe-is-zero} implies that there
    exists a slice regular function $f_1$ with
    \[
    D_1=\{q\in\Omega_G: f_1(q)=0\}.
    \]
    Due to Proposition~\ref{star-zero} the function $f=f_0*f_1$
    satisfies
    \[
    \{q\in\Omega_G: f(q)=0\}=D_1\cup D_0=D.
    \]
    Finally, let $F_i$ be the stem function for $f_i$
    ($i\in\{0,1\}$.
    Due to Theorem~\ref{prescribe-is-zero}, $(iii)$, we have
    $F_1(q)\in\Ac^*$ for all $q$ in a non-empty open subset of $G$.
    On the other hand $F_0(q)\in\C^*\subset\Ac^*$ for
    all $q$ in a non-empty Zariski open subset of $G$.
    A non-empty Zariski open subset intersects every non-empty
    open subset. Hence we may choose $p\in G$ such that
    both $F_0(p)$ and $F_1(p)$ are contained in $\Ac^*$.
    Now the stem function $F$ of $f=f_0*f_1$ equals $F_0\cdot F_1$.
    Hence
    $F(p)\in\Ac^*$ due to Corollary~\ref{inv-prod}.
  \end{proof}

    \section{The structure of the zero-set}
  Since we have seen that every adapted set may be realized as the
  zero set of a slice regular function
  (Theorem~\ref{thm-adapted}),
  it is natural to ask about the
  converse:
  {\em Is the zero-set of a slice regular function
  which is not constantly zero necessarily an adapted
  set?}

  Without additional assumptions
  the answer is NO, as can be seen by an example of
  Ghiloni and Perotti (\cite{GP}, Remark 12,
  see also \cite{GPS-TAMS}, Example~4.13),
  which we now describe in our words:

  \begin{example-nr}[Ghiloni-Perotti]\label{not-adapted}
  Let $A=\H$, $\Omega_G=\H\setminus\R$, fix $J\in\S$ and define
  \[
  f(q)=\frac{Im(q)}{|Im(q)|}-J.
  \]
  In other words:
  \[
  f(x+yH)=H-J\quad \text{ for }x\in\R, y\in\R^+,H\in\S_\H,
  \]
  This is a slice regular function with stem function
  \[
  F(z)=\begin{cases}
  i-J & \text{ if $\Im(z)>0$}\\
  -i-J & \text{ if $\Im(z)<0$}\\
  \end{cases}
  \]
  The zero set of $f$ is
  \[
  \C_J^+=\{x+yJ:x,y\in\R, y>0\}.
  \]
  This is not an adapted set.
  \end{example-nr}

  Another example is provided by
  the Clifford algebra $\R_3$ (see e.g.~\cite{BDMW}
  for more on $\R_3$).
  \begin{example}
    On the quadratic cone of
    $\R_3\simeq\H\oplus\H$ a slice regular function $f$
    is given by $(q_1,q_2)\mapsto (0,q_2-J)$.

    Recall
    \[
    Q_{\R_3}=\{(q_1,q_2)\in\H: N(q_1)=N(q_2), Tr(q_1)=Tr(q_2)\}.
    \]

    Then
    \[
    \{(q_1,q_2)\in Q_{\R_3}:f(q_1,q_2)=0\}
    =\S_{\H}\times\{J\}
    \]
    which is not an adapted set.
  \end{example}

  On the other hand, for division algebras
  (i.e., for $\H$ and $\Oct$) there is a positive result
  under an additional assumption (namely $G\cap\R\ne\{\}$):
  
  \begin{theorem}\label{thm-mixed}
    Let $A$ be a division $*$-algebra (i.e., $A=\C$, $A=\H$ or $A=\Oct$).
    Let $G$ be a symmetric domain in $\C$
    with $G\cap\R\ne\{\}$
        and let $\Omega_G$ be the
    associated axially symmetric domain in $Q_A=A$.%
    \footnote{For $A$ being isomorphic to $\C$, $\H$ or $\Oct$
      we have $Q_A=A$.}

    Then for every slice regular function
    $f:\Omega_G\to\H$ 
    the zero set
    \[
    D=\{q\in\Omega_G: f(q)=0\}
    \]
    is an {\em adapted} subset (in the sense of
    Definition~\ref{def-adapted}) of $\Omega_G$.
  \end{theorem}

  \begin{proof}
    See \cite{camshaft}, \cite{GP}, \cite{GSSMono}.
  \end{proof}

    \begin{remark}
   Ghiloni and Perotti \cite{GP}
  introduced the notion of ``admissible functions''.

  For such ``admissible functions'' they prove that the
  zero set of a slice regular function $f$ on an axially
  symmetric domain $\Omega\subset Q_A\subset A$ is what we
  call an ``adapted set'', even if $A$ is not a division algebra.
  \end{remark}
    
      In general, if we require that the stem function has
      an invertible function value somewhere, we can show that
      the zero set must be contained in an adapted set.
    
  \begin{proposition}\label{zero-sub-adapted}
    Let $A$ be an alternative finite-dimensional real $*$-algebra and $f$ a slice regular
    function on an axially symmetric  domain $\Omega_G$ such that
    there is a point $p\in G$ in which the stem function has
    an invertible value (i.e.~$F(p)\in\Ac^*$).

    Then the zero set $V(f)=\{w\in \Omega_G: f(w)=0\}$
    is contained in an adapted subset of $\Omega_G$.
  \end{proposition}

  \begin{proof}
    The set of invertible elements in $\Ac$ is Zariski open
    (Corollary~\ref{unit-open}). Hence
    \[
    \Lambda=\{z\in G: F(z)\not\in\Ac^*\}
    \]
    is either discrete or equal to $G$. But we assumed that there
    is one point $p$ for which $F(p)\in\Ac^*$.
    Hence $\Lambda$ is discrete and
    \[
    D=\{x+yI:x,y\in\R,x+yi\in\Lambda, I\in\S_A\}
    \]
    is an adapted set.
    Corollary~\ref{z-c1} and
    the implication ``$(i)\implies(v)$'' of Proposition~\ref{alt-eq} 
    imply that $ \{w\in\Omega_G:f(w)=0\}\subset D$.
  \end{proof}
  \begin{remark}
    Recall: Given a slice regular function $f$ and a pseudosphere
    $x+y\S$, there is an affine-linear function $L$ such that
    $f(x+yI)=L(x+yI)\ \forall I\in\S$.

    Thus $\{f=0\}\cap (x+y\S)$ is not arbitrary, it necessarily concides
    with $\{x+yI: I\in\S, (x+yI)a+b=0\}$ for some $a,b\in A$.

    Therefore the above proposition provides a fairly complete
    picture of possible zero sets for slice regular functions.
  \end{remark}

  \chapter{Main Results: (Non-)Extension}

\section{Topological Preparations: Discrete sets}

\begin{lemma}\label{L16-1}
  Let $\pi:X\to Y$ be a continuous map between
  metric topological
  spaces which is nowhere locally constant.

  Then for every $\epsilon,\delta>0$, $p\in X$ there is an element
  $\tilde p\in X$ such that
  \[
  0<d_X(p,\tilde p)<\delta,\quad
  0<d_Y(\pi(p),\pi(\tilde p))<\epsilon
  \]
\end{lemma}

\begin{proof}
  Assume the contrary.
  Then we have a triple $(\epsilon,\delta,p)$ such that $\pi$ is
  constant on the open neighborhood
  \[
  \{x\in X:d_X(x,p)<\delta\}\cap\pi^ {-1}\left(\{y\in Y:d_Y(\pi(p),y)<\epsilon\}
  \right)
  \]
  of $p$, contradicting the assumption that $\pi$ is nowhere locally
  constant.
\end{proof}

\begin{lemma}\label{make-inj}
  Let $\pi:X\to Y$ be a continuous map between metric topological spaces
  which is nowhere locally constant. Assume that
    $D=\{p_k:k\in\N\}$ is a countable subset of  $X$.

  Then there exists a sequence $q_k$ in $X$ with
  $d_X(p_k,q_k)<2^{-k}\ \forall k$ such that $\pi$ restricts
  to an {\em injective} map from $\tilde D=\{q_k:k\in\N\}$ to $Y$.
\end{lemma}

\begin{proof}
  We construct $\tilde D$ recursively, starting with $q_1=p_1$.

  Given $q_1,\ldots,q_n$ we choose $q_{n+1}$ as follows:

  If $\pi(p_{n+1})\not\in\pi(\{q_1,\ldots,q_n\})$, we set $q_{n+1}=p_{n+1}$.

  Now assume
  \[
  \exists m\in\{1,\ldots,n\}:
  \pi(p_{n+1})=\pi(q_m).
  \]

    Note: By construction, $\pi(q_1),\ldots,\pi(q_n)$ are pairwise
    distinct. Hence there can be only one such $m$ and
    $\pi(p_{n+1})\ne q_k$ for all
    $k\in\{1,\ldots,n\}\setminus\{m\}$
  
  Fix $\delta>0$ such that
  \[
  \forall k\in\{1,\ldots,n\}\setminus\{m\}:\ d_Y(\pi(p_{n+1}),\pi(q_k))>\delta.
  \]
  Because $\pi$ is nowhere locally constant,
  using Lemma~\ref{L16-1} we can find an element $q_{n+1}$
  such that
  \[
  0<d_Y(\pi(p_{n+1}),\pi(q_{n+1}))<\delta\quad
  \text{and }\quad
  d_X(p_{n+1},q_{n+1})<2^{-(n+1)}.
  \]
  By construction $q_{n+1}$ has the desired properties.
\end{proof}

\begin{lemma}\label{keep-acc}
  Let $X$ be a metric space with subsets $D=\{p_k:k\in\N\}$
  and $\tilde D=\{q_k:k\in\N\}$ such that
  $\lim_{k\to+\infty} d_X(p_k,q_k)=0$.

  Then every accumulation point of $D$ is also an accumulation
  point of $\tilde D$ and vice versa.
\end{lemma}

\begin{proof}
  A point $p\in X$ is an accumulation point of $D$ if there is a sequence
  $n_k$ in $\N$ such that $\lim_{k\to+\infty} d_X(p_{n_k},p)=0$.
  In combination with $\lim_{k\to+\infty} d_X(p_k,q_k)=0$ and the
  triangle inequality this yields the assertion.
\end{proof}

\begin{corollary}\label{L16.3}
  Let $\pi:X\to Y$ be a continuous map between metric topological spaces
  which is nowhere locally constant. Assume that
    $D$ is a countable subset of  $X$.

  Then there exists a countable subset $\tilde D$ in $X$ such that
  \begin{enumerate}
    \item $D$ and $\tilde D$ have the same accumulation points in $X$,
  \item
    $\pi$ restricts to an injective map $\pi|_{\tilde D}:\tilde D\to Y$
  \end{enumerate}
\end{corollary}

\begin{proof}
  Combine Lemma~\ref{make-inj} and Lemma~\ref{keep-acc}.
\end{proof}

The following topological result will be applied to the case where
$X$ is the quadratic cone $Q_A$ of an alternative finite-dimensional real $*$-algebra
and $Y=\{z\in\C:\Im ( z)\ge 0\}$ with $\pi(x+yI)=x+|y|i$ for
  $x,y\in\R, I\in\S_A$.
  
\begin{proposition}\label{P16-4}
  Let $X$ be a locally compact metric topological space,
  $Y$ a metric topological space,
  $\pi:X\to Y$ a 
  continuous map, $U\subset Y$ an open subset.
  
  Then $\tilde\Omega=\pi^{-1}(U)$ contains a discrete subset $D$
  such that
  \begin{enumerate}
  \item
    $\pi$ maps $D$ 
    onto a discrete subset of $U$.
  \item
    Every point in $\partial\tilde\Omega$ is an accumulation point of $D$.
  \end{enumerate}
\end{proposition}

\begin{proof}
  If $U=\{\}$ or $U=Y$, the assertion is trivially
  satisfied with $D=\{\}$.
  Hence we assume $\{\}\ne U\ne Y$.
  
  Fix an exhaustion function $\rho:X\to\R^+$.%
  \footnote{The existence of an exhaustion function follows from
    $X$ being a locally compact metric space.}
  For $x\in X$ we define
  \[
  \delta(x)=
    \mathop{d}{}_Y(\pi(x),\partial U)=\inf_{z\in\partial U} d_Y(\pi(x),z).
  \]
  For every $n\in\N$ we consider the set
  \[
  C_n=\left\{x\in X:\rho(x)\le n, \frac 1n\ge \delta(x)\ge \frac1{n+1}
  \right\}
  \]
  Since $\rho$ is an exhaustion function, $C_n$ is a closed subset
  of the compact set $\{x:\rho(x)\le n\}$. Hence $C_n$ is compact
  and we may choose
  finitely many points $c_{n,1},\ldots,c_{n,r_n}$ such that
  
  \begin{equation}\label{cn-cover}
    C_n\subset\bigcup_{j\in\{1,\ldots,r_n\}}\left\{ x\in X:
    d_X\left(x,c_{n,j}\right)<\frac 1{n}
  \right\}
  \end{equation}
  
  We define $D$ as the set of all $ c_{n,j}$.

  By construction (observe that $C_k\cap C_m=\{\}$ unless $|k-m|\le 1$),
  \[
  \left\{x\in D: \frac 1n\ge \delta(x)\ge\frac 1{n+1}\right\}
  \]
  is finite for every $n\in\N$.

  Given a point $y\in U$,
  we have $d_Y(y,\partial U)>0$.
  Assume $d_Y(y,\partial U)\in
  ]\frac 1{n+1},\frac 1n]$.
  Then the set
  \[
  W=\left\{q\in Y:\frac 1{n-1}>d_Y(q,\partial U) >\frac 1{n+1}\right\}
  \]
  is an open neighborhood of $y$ such that
  \[
  \pi(D)\cap W\subset \pi(D\cap(C_{n-1}\cup C_n))
  \]
  is finite.
  Hence $\pi(D)$ is discrete in $U$ and $D$ is discrete
  in $X$.

  Finally, we have to show that $\partial \tilde\Omega$ is contained
  in the closure of $D$ in $X$.

  Let $p\in\partial\tilde\Omega$. Let $N>\rho(p)$.
  Choose a sequence $p_k\in\tilde\Omega$
  with $\lim_{k\to+\infty}p_k=p$.
  We require that $\rho(p_k)<N\ \forall k$.
  Since
  \[
  \lim_{k\to+\infty}\pi(p_k)=\pi(p)\in\partial U,
  \]
  we may furthermore
  require that $\delta(p_k)<1\ \forall k$.
  Then we define a map $k\mapsto n(k)$ by stipulating
  \[
  n(k)\in\N,\quad \frac1{n(k)}\ge \delta(p_k)>\frac{1}{n(k)+1}.
  \]
  In other words:
  \[
  n(k)=\left\lfloor \frac 1 {\delta(p_k)} \right\rfloor.
  \]
  Evidently
  $\lim_{k\to+\infty} \pi(p_k)=\pi(p)\in\partial U$ implies
  $\lim_{k\to+\infty} \delta(p_k)=0$ which in turn implies
  $\lim_{k\to+\infty} n(k)=+\infty$.
  By construction $p_k\in C_{n(k)}$ if $\rho(p_k)\le n(k)$.
  In particular, $p_k\in C_{n(k)}$ if $n(k)>N$.
  Fix $k_0\in\N$ such that
  \[
  n(k)> N\ \forall k\ge k_0.
  \]

  As a consequence, we have
  \[
  p_k\in C_{n(k)}\ \forall k\ge k_0.
  \]
  Due to \eqref{cn-cover} for every element
    $q\in C_{n(k)}$ there is an element $\lambda\in D$ with
    $d_X(q,\lambda)<\frac 1{n(k)}$.
  Thus   
  for each $k\ge k_0$ we can
  choose $\lambda_k\in D$ such that
  \[
  d_X(p_k,\lambda_k)<\frac{1}{n(k)}.
  \]
  Now
  \[
  d_X(p,\lambda_k)\le   d_X(p,p_k) +   d_X(p_k,\lambda_k)<
  d_X(p,p_k)+\frac 1{n(k)}
  \ \forall k\ge k_0
  \]
  and, using $\displaystyle\lim_{k\to+\infty} p_k=p$
  and $\displaystyle\lim_{k\to+\infty} n(k)=+\infty$, we obtain
  $\displaystyle\lim_{k\to+\infty} \lambda_k=p$.
\end{proof}

\begin{corollary}\label{exist-discrete}
  Let $A$ be an alternative finite-dimensional
  real $*$-algebra with quadratic cone $Q_A$
  and pseudosphere $\S_A$.
  
  Let $\Omega_G\subset Q_A$ be an axially symmetric domain in $A$
  associated to a symmetric domain $G$ in $\C$.
  
  Then there exists a discrete subset $D$ of $\Omega_G$ such that
  \begin{enumerate}
  \item
   For every $x,y\in\R$ the intersection
   $\left(x+y\S_A\right)\cap D$ contains at most one point.
   \item
   The closure of $D$ in $Q_A$ contains the boundary $\partial\Omega_G$
   of $\Omega_G$ in $Q_A$.
 \item
   \[
   \tilde D=\{x+yi: \exists J\in\S_A: x+yJ\in D\}
   \]
   is discrete in $G$.
  \end{enumerate}
\end{corollary}

\begin{proof}
  As a first step, we will show that there exists a discrete subset $D_0$
  satisfying $(ii)$ and $(iii)$.
  Later we will demonstrate that we may modify the discrete set $D_0$
  in such
  a way that the modified discrete set $D$ satisfy in addition condition $(i)$.

  We regard the map $\pi:Q_A\to H=\{z\in\C:\Im(z)\ge 0\}$ given as
  \[
  \pi:x+yI\mapsto x+|y|i\quad (x,y\in\R,I\in \S_A).
  \]
  We observe that $I\in\S_A\implies I=-I^c$ and $II^c=1$.
  Hence
  \[
  \forall x,y\in\R, I\in\S_A:\ \Tr(x+yI)=2x,\quad N(x+yI)=x^2+y^2.
  \]
  Therefore
  \[
  \forall w=x+yI\in Q_A: x+|y|i=\frac 12\Tr(w)+i\sqrt{N(w)-\frac 14\Tr(w)^2}.
  \]
  It follows that $\pi:Q_A\to H$ extends to a continuous map
  $\bar\pi$ defined on all of $\overline{Q_A}$, namely
  \[
  \bar\pi:w\mapsto\frac 12\Tr(w)+i\sqrt{N(w)-\frac 14\Tr(w)^2}
  \]
  (Note that $N(w)-\frac 14\Tr(w)^2\ge 0$ for $w\in Q_A$, implying
  $N(w)-\frac 14\Tr(w)^2\ge 0$ for $w\in\overline{Q_A}$.
  Hence $\sqrt{N(w)-\frac 14\Tr(w)^2}$ is a well defined element in
  $\R^+_0$.)

  Let $\Omega'_G=\{w\in\overline{Q_A}:\bar\pi(w)\in G\}$.
  As a next step, we apply Proposition~\ref{P16-4} with $X=\overline{Q_A}$,
  $\tilde\Omega=\Omega'_G$, $Y=H$ and $U=G$, obtaining a countable
  subset $D'$ in $\Omega'_G$
  such that
  \begin{itemize}
  \item
    $D'$ is discrete in $\Omega'_G$.
  \item
    $\bar\pi$ maps $D'$ onto a discrete
    subset of $G$.
  \item
    The closure of $D'$ in $\overline{Q_A}$ contains the boundary
    $\partial\Omega'_G$ of $\Omega'_G$ in $\overline{Q_A}$.
  \end{itemize}

  Let $v$ be a point in the boundary of $\Omega_G$ in $Q_A$,
  with $\lim v_k=v$ for a sequence $v_k$ in $\Omega_G$.
  By construction we have $\Omega_G'\cap Q_A=\Omega_G$.
  Hence $v_k\in\Omega_G'$, but $v\not\in\Omega_G'$.
  It follows that $v$ is contained in the
   boundary
   $\partial\Omega'_G$ of $\Omega'_G$ in $\overline{Q_A}$.
   Thus $D_0$ satisfies condition $(ii)$.

  If $D'\setminus(\Omega_G\cap D')$ is finite, we set $D_0=D'\cap\Omega_G$
    and we are done with the construction of
    $D_0$.

  This leaves the case where $D'\setminus(\Omega_G\cap D')$ is infinite.
  We enumerate this set
  \[
  D'\setminus(\Omega_G\cap D')
  = \{ p_k:k\in\N\}.
  \]
  Since every $p_k$ is contained in
  $\overline{Q_A}\setminus Q_A$, we may choose points
  $q_k\in\Omega_G=Q_A\cap\Omega'_G$
  such that
  \[
  d_X(p_k,q_k)<2^{-k}.
  \]
  Define
  \[
  D_0=(D'\cap\Omega_G)\cup\{q_k:k\in\N\}.
  \]
  
  Then $D_0$ satisfies properties $(ii)$ and $(iii)$, as desired.

  Finally, using Corollary~\ref{L16.3} we may
  modify $D_0$ in order to obtain a discrete set $D$ for which
  $\pi|_D$ is injective,
  without losing the other properties, i.e., $D$ (like $D_0$) satisfies
  conditions $(ii)$ and $(iii)$.

    Due to the definition of $\pi$, the fibers $\pi^{-1}(x+|y|i)$
    of the map $\pi|_D:D\to H$ are precisely the sets
    $\left(x+y\S_A\right)\cap D$. Hence
    injectivity of $\pi|_D:D\to H$ is equivalent to condition $(i)$.
  
\end{proof}

\begin{remark}
  If the quadratic cone $Q_A$ is not closed in $A$, then $Q_A$ is not necessarily
  locally compact. Therefore we can not apply Proposition~\ref{P16-4}
  directly to the quadratic cone $Q_A$
  and instead have to deal with the closure $\overline{Q_A}$.
  As a closed subset in a finite-dimensional real vector space,
  $\overline{Q_A}$ is locally compact.
\end{remark}

\section{Constructing Non-extendible Functions}

\begin{theorem}\label{no-extend}
  Let $A$ be a finite-dimensional real alternative
  $*$-algebra with quadratic cone $Q_A$.
  
  Let $\Omega_G$ be an axially symmetric domain in $Q_A$
  associated to a symmetric domain $G$ in $\C$.
  
  Then there exists a slice regular function $f$ on $\Omega_G$
  such that
  \begin{enumerate}
  \item
    $D=\{q\in\Omega_G:f(q)=0\}$ is discrete in $\Omega_G$.
  \item
    $D\cup\partial\Omega_G$ is the closure of $D$ in $Q_A$, i.e.,
    {\em every} boundary point of $\Omega_G$ is an accumulation point
    of $D$.
  \item
    There is a point $z\in G$ such that $F(z)\in\Ac^*$ for
    the stem function $F$ of $f$.
  \end{enumerate}
\end{theorem}
  
\begin{proof}
  Due to Corollary~\ref{exist-discrete} we can find a discrete subset $D$
  of $\Omega_G$ such that
  \begin{enumerate}
  \item
   For every $x,y\in\R$ the intersection
   $\left(x+y\S_A\right)\cap D$ contains at most one point.
 \item
   The closure of $D$ in $Q_A$ contains the boundary $\partial\Omega_G$
   of $\Omega_G$ in $Q_A$.
  \end{enumerate}

  Thanks to Theorem~\ref{thm-adapted}
  there exists a slice regular function
  on $\Omega_G$ which has isolated zeroes at every point of $D$
  and no other zeroes and satisfies $(iii)$.
\end{proof}

\begin{corollary}\label{c-no-extend}
  Let $A$ be a finite-dimensional real alternative
  $*$-algebra with quadratic cone $Q_A$.
  
  Let $\Omega_G$ be an axially symmetric domain in $Q_A$.
  
  Then there exists a slice regular function $f$ on $\Omega_G$
  with discrete set of zeroes $D=\{q\in\Omega_G:f(q)=0\}$
  which can not be extended
  (as a slice regular function)
  to any larger axially symmetric domain $\Omega_H$
  ($\Omega_G\subsetneq \Omega_H\subset Q_A$).
\end{corollary}

  \begin{proof}
  Let $\Omega_H$ be an axially symmetric domain
   with
   $\Omega_G\subsetneq \Omega_H$.
   Let $G$ and $H$ be the symmetric domains in $\C$
   to which $\Omega_G$
   and $\Omega_H$ are associated.
   
  Let $f$ be a slice regular function on $\Omega_G$ as
  provided by theorem~\ref{no-extend},
  with stem function $F:G\to\Ac$.
  Assume that $f$ extends to a slice regular function $\tilde f$
  on $\Omega_H$ with stem function $\tilde F:H\to\Ac$.
  Choose a point $p\in \Omega_H\cap\partial \Omega_G$.
  As usual, let $\pi$ denote the map from $Q_A$ to $\C$
  which maps each $x+yI$ to $x+|y|i$.
  By the choice of $f$
    ($(ii)$ of Theorem~\ref{no-extend}),
    $p$ is an accumulation point of $D$.
  Therefore $\pi(p)$ is an accumulation point of $\pi(D)$.
  It follows that $\pi(D)$ is not discrete in $H$.
  Thus $\pi\left(\{q\in\Omega_H:\tilde f(q)=0\}\right)$
  is not discrete in $H$.

  On the other hand, consider
  \[
  H^*=\{w\in H: \tilde F(w)\in\Ac^*\}
  \]
  By construction of the function $f$
  (see condition $(iii)$ in Theorem~\ref{no-extend}),
  there is a point $z\in G$ with $F(z)=\tilde F(z)\in\Ac^*$.
  Hence the set $H^*$ is not empty. Corollary~\ref{unit-open} implies that
  $H\setminus H^*$ is a discrete subset of $H$ and
  Corollary~\ref{z-c1} implies that $\pi(D)\cap H^*$
  is empty. It follows that $\pi(D)$ must be discrete in $H$.
  Contradiction!
\end{proof} 

\begin{remark}
  If we drop the demand that the zero set of the slice regular funtion
  should be discrete, there is an easy way to obtain a non-extendible
  function: Let $F_0:G\to\C$ be a holomorphic function whose zeroes
  accumulate everywhere
  at the boundary of $G$ in $\C$
  and define $F(z)=F_0(z)\overline{F_0(\bar z)}$.
  Then $F$ is a stem function of a {\em slice preserving}
  regular function $f$ which evidently does not extend to any
  larger axially symmetric domain.
\end{remark}
\nocite{BW1}

  \bibliography{auto}
\bibliographystyle{alpha}

\end{document}